\documentclass[11pt,reqno]{amsart}
\usepackage[margin=3cm]{geometry}
\usepackage{color}               
\usepackage{faktor}
\usepackage{tikz}
\usepackage{tikz-cd}
\usepackage{spverbatim}
\usepackage{float}
\usepackage[pdftex,hyperfootnotes=false,colorlinks=false,hypertexnames=false]{hyperref}
\usepackage{makecell}
\usepackage{multicol}

\usepackage{amsmath,amssymb,anyfontsize,enumerate,stmaryrd,mathtools, thmtools,tikz,txfonts,nccmath}
\usepackage[oldstyle]{libertine}%
\usepackage[T1]{fontenc}
\usepackage[final]{microtype}
\usepackage[utf8]{inputenc}
\usepackage{csquotes}
\makeatletter
\renewcommand*\libertine@figurestyle{LF}
\makeatother
\usepackage[libertine,libaltvw,liby]{newtxmath}
\makeatletter
\renewcommand*\libertine@figurestyle{OsF}
\makeatother
\usepackage[noerroretextools,backend=biber,style=alphabetic,maxalphanames=5,giveninits,sorting=nyt,%
maxbibnames=99]{biblatex}
\usepackage{tikz-cd}
\usepackage[noabbrev,capitalise]{cleveref}
\usepackage{autonum}
\usepackage{braket}

\theoremstyle{plain}
    \newtheorem{theorem}{Theorem}[section]
    \newtheorem{construction/theorem}[theorem]{Construction/Theorem}

    \newtheorem{proposition}[theorem]{Proposition}
    
\theoremstyle{definition}
    \newtheorem{remark}[theorem]{Remark}
    
    \newtheorem{example}[theorem]{Example}
    \newtheorem{definition}[theorem]{Definition}

\newcommand{\comment}[1]{}

\title{Universal piecewise polynomiality for counting curves in toric surfaces}
\author[M.~A.~Hahn]{Marvin Anas Hahn}
\address{M.~A.~Hahn: School of Mathematics, 17 Westland Row, Trinity College Dublin, Dublin 2, Ireland}
\email{hahnma@maths.tcd.ie}
\author[V.~Reda]{Vincenzo Reda}
\address{V.~Reda: School of Mathematics, 17 Westland Row, Trinity College Dublin, Dublin 2, Ireland}
\email{redav@tcd.ie}


\thanks{\textbf{Acknowledgements.} We thank Hannah Markwig for many valuable discussions.}

\begin{document}
\begin{abstract}
Inspired by piecewise polynomiality results of double Hurwitz numbers, Ardila and Brugallé introduced an enumerative problem which they call 
\textit{double Gromov--Witten invariants of Hirzebruch surfaces} in \cite{ardila2017double}. These invariants serve as a two-dimensional analogue and satisfy a similar piecewise polynomial structure. More precisely, they introduced the enumeration of curves in Hirzebruch surfaces satisfying point conditions and tangency conditions on the two parallel toric boundaries. These conditions are stored in four partitions and the resulting invariants are piecewise polynomial in their entries. Moreover, they found that these expressions also behave polynomially with respect to the parameter determining the underlying Hirzebruch surfaces. Based on work of Ardila and Block \cite{ardila2013universal}, they proposed that such a polynomiality could also hold while changing between more general toric surfaces corresponding to $h$-transverse polygons. In this work, we answer this question affirmatively. Moreover, we express the resulting invariants for $h$-transverse polygons as matrix elements in the two-dimensional bosonic Fock space.
\end{abstract}

\maketitle

\tableofcontents

\section{Introduction}\label{intro}
Double Gromov--Witten invariants of Hirzebruch surfaces were introduced by Ardila and Brugallé in \cite{ardila2017double} as a two-dimensional generalisation of so-called \textit{double Hurwitz numbers}. While double Hurwitz numbers count branched morphisms to the Riemann sphere with two relative conditions describing the ramification data, this two-dimensional analogue enumerates curves in a Hirzebruch surface passing through a certain number of points and having prescribed tangency order at two distinguished toric boundary divisors .\\
Double Hurwitz numbers admit a piecewise polynomial structure and tropical geometry has proved to be a powerful framework in the study of the polynomiality of Hurwitz numbers \cite{cavalieri2010tropical,cavalieri2011wall,hahn2022tropical,hahn2020wall,hahn2022twisted,fitzgerald2023combinatorics}.
A key idea in these tropical approaches is to express Hurwitz numbers as a weighted enumeration of abstract graphs. For the enumeration of curves in Hirzebruch surfaces, \textbf{Floor diagrams} provide a two-dimensional analogue and were introduced in \cite{brugalle2008floor,brugalle2016floor}. They allow to translate the enumerative problem in terms of a weighted count of abstract decorated graphs. Motivated by this and by storing the tangency orders in partitions, Ardila and Brugallé used tropical geometry to show that double Gromov--Witten invariants of Hirzebruch surfaces admit a piecewise polynomial behaviour in the entries of those partitions that closely mirrors the piecewise polynomiality of double Hurwitz numbers in the entries of the ramification profiles. Moreover, they observed that the piecewise polynomiality extends when moving between different Hirzebruch surfaces. To be precise, recall that Hirzebruch surfaces $\mathbb{F}_k$ live in a one-dimensional family indexed by natural numbers $k\ge0$. The polynomials describing double Gromov--Witten invariants of Hirzebruch surfaces are also polynomial in the parameter $k$. Based on this observation, Ardila and Brugallé asked in \cite[Section 7]{ardila2017double} whether this polynomial interpolation between Hirzebruch surfaces extends to enumerations of curves in more complicated toric surfaces -- an analogous statement for Severi degrees of toric surfaces was derived in \cite{ardila2013universal,liu2018severi}. In this work, we answer this question affirmatively for the important family of toric surfaces corresponding to \textit{$h$-transverse polygons}.

\begin{definition}\label{def1}
A polygon $P$ is said to be $h$-transverse if every vertex has integer coordinates and every edge has slope $0$, $\infty$ or $\frac1k$, with $k\in\mathbb{Z}\setminus\{0\}$.
\end{definition}

The enumerative geometry of these toric surfaces may be studied via the aforementioned floor diagrams, making this class a suitable choice for our purposes. Indeed, in this work we take the same approach employed in \cite{ardila2017double}.

\subsection{Double Gromov--Witten invariants for $h$-transverse polygons}

We now set-up our counting problem of \textit{double Gromov--Witten invariants for $h$-transverse polygons}. To begin with, we parametrise $h$-transverse polygons.\\
Let $P$ be an $h$-transverse polygon. We will always assume that $P$ has two edges of slope $0$, a top edge and a bottom edge. We denote by $d^t>0$ the lattice length of the top edge and $d^b>0$ the lattice length of the bottom edge. Moreover, we record the direction of the edges on the left of $P$ by $(c_1^l,-1),\dots,(c_m^l,-1)$ in counterclockwise order. Similarly, we record the direction of the edges on the right of $P$ by $(c_1^r,-1),\dots,(c_n^r,-1)$ in a clockwise manner. Note that $c_i^l$ and $c_j^r$ are integers. Moreover, we have $c_1^l<\dots<c_m^l$ and $c_1^r>\dots>c_n^r$. In addition, we denote the lattice length of the side corresponding to $c_i^l$ by $d_i^l$ and the lattice length of $c_j^r$ by $d_j^r$. We define partitions $\textbf{c}^l=(c_1^l,\dots,c_m^l)$, $\textbf{c}^r=(c_1^r,\dots,c_n^r)$, $\textbf{d}^l=(d_1^l,\dots,d_m^l)$ and $\textbf{d}^r=(d_1^r,\dots,d_m^r)$. Denoting $\textbf{c}=(\textbf{c}^r;\textbf{c}^l)$ and $\textbf{d}=(d^t;\textbf{d}^r;\textbf{d}^l)$, we see that $\textbf{c}$ and $\textbf{d}$ completely determine the $h$-transverse polygon $P$ and we will denote $P=P(\textbf{c},\textbf{d})$. Moreover, we denote the toric surface corresponding to $P(\textbf{c},\textbf{d})$ by $S(\textbf{c})$.

\begin{remark}\label{rmk2}
    Since the normal fan of $P$ is balanced, we obtain the following relations:
    \begin{equation}
        d^t+\sum_{i=1}^nc_i^rd_i^r-d^b-\sum_{j=1}^mc_j^ld_j^l=0
    \end{equation}
    \begin{equation}
        \sum_{i=1}^nd_i^r=\sum_{j=1}^md_j^l.
    \end{equation}
\end{remark}

\begin{example}
\label{example2}
    An important class of $h$-transverse polygons are the polygons corresponding to Hirzebruch surfaces, that are obtained from this construction by setting $n=m=1$, $c^r=k\geq0$ and $c^l=0$, see \cref{fig-hirzebruch}

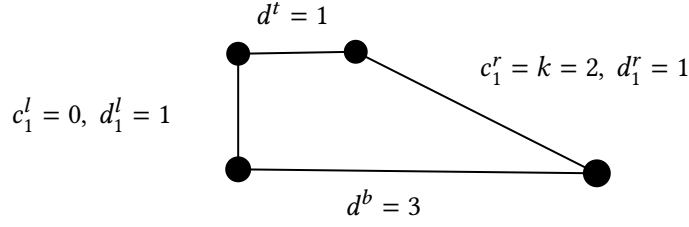
\begin{figure}
    \centering
    \tikzset{every picture/.style={line width=0.75pt}} 

\begin{tikzpicture}[x=0.75pt,y=0.75pt,yscale=-1,xscale=1]

\draw   (210,91) -- (331,152) -- (151,150) -- (151,92) -- cycle ;
\draw  [fill={rgb, 255:red, 0; green, 0; blue, 0 }  ,fill opacity=1 ] (145.5,92) .. controls (145.5,88.96) and (147.96,86.5) .. (151,86.5) .. controls (154.04,86.5) and (156.5,88.96) .. (156.5,92) .. controls (156.5,95.04) and (154.04,97.5) .. (151,97.5) .. controls (147.96,97.5) and (145.5,95.04) .. (145.5,92) -- cycle ;
\draw  [fill={rgb, 255:red, 0; green, 0; blue, 0 }  ,fill opacity=1 ] (204.5,91) .. controls (204.5,87.96) and (206.96,85.5) .. (210,85.5) .. controls (213.04,85.5) and (215.5,87.96) .. (215.5,91) .. controls (215.5,94.04) and (213.04,96.5) .. (210,96.5) .. controls (206.96,96.5) and (204.5,94.04) .. (204.5,91) -- cycle ;
\draw  [fill={rgb, 255:red, 0; green, 0; blue, 0 }  ,fill opacity=1 ] (145,150) .. controls (145,146.69) and (147.69,144) .. (151,144) .. controls (154.31,144) and (157,146.69) .. (157,150) .. controls (157,153.31) and (154.31,156) .. (151,156) .. controls (147.69,156) and (145,153.31) .. (145,150) -- cycle ;
\draw  [fill={rgb, 255:red, 0; green, 0; blue, 0 }  ,fill opacity=1 ] (324.5,152) .. controls (324.5,148.41) and (327.41,145.5) .. (331,145.5) .. controls (334.59,145.5) and (337.5,148.41) .. (337.5,152) .. controls (337.5,155.59) and (334.59,158.5) .. (331,158.5) .. controls (327.41,158.5) and (324.5,155.59) .. (324.5,152) -- cycle ;

\draw (36,112.4) node [anchor=north west][inner sep=0.75pt]    {$c_{1}^{l} =0,\ d_{1}^{l} =1$};
\draw (271,91.4) node [anchor=north west][inner sep=0.75pt]    {$c_{1}^{r} =k=2,\ d_{1}^{r} =1$};
\draw (159,64.4) node [anchor=north west][inner sep=0.75pt]    {$d^{t} =1$};
\draw (204,158.4) node [anchor=north west][inner sep=0.75pt]    {$d^{b} =3$};

\end{tikzpicture}

    \caption{The polygon corresponding to the Hirzebruch surface $\mathbb{F}_2$.}
    \label{fig-hirzebruch}
\end{figure}

\end{example}

\begin{example}\label{example1}
    For $\textbf c=(\textbf c^r;\textbf c^l)=(3,1,-3;-1,0)$ and $\textbf d=(d^t; \textbf d^r; \textbf d^l)=(2;1,2,1;2,2)$, one obtains the $h$-transverse polygon in \cref{fig-htransverse}.\\

    \begin{figure}
    \tikzset{every picture/.style={line width=0.75pt}} 

\begin{tikzpicture}[x=0.75pt,y=0.75pt,yscale=-1,xscale=1]

\draw  [color={rgb, 255:red, 0; green, 0; blue, 0 }  ,draw opacity=1 ][line width=1.0]  (300,61) -- (390,91) -- (451,150) -- (361,181) -- (181,181) -- (180,120) -- (240,61) -- cycle ;
\draw  [fill={rgb, 255:red, 0; green, 0; blue, 0 }  ,fill opacity=1 ] (234,61) .. controls (234,57.69) and (236.69,55) .. (240,55) .. controls (243.31,55) and (246,57.69) .. (246,61) .. controls (246,64.31) and (243.31,67) .. (240,67) .. controls (236.69,67) and (234,64.31) .. (234,61) -- cycle ;
\draw  [fill={rgb, 255:red, 0; green, 0; blue, 0 }  ,fill opacity=1 ] (294,60.75) .. controls (294,57.3) and (296.8,54.5) .. (300.25,54.5) .. controls (303.7,54.5) and (306.5,57.3) .. (306.5,60.75) .. controls (306.5,64.2) and (303.7,67) .. (300.25,67) .. controls (296.8,67) and (294,64.2) .. (294,60.75) -- cycle ;
\draw  [fill={rgb, 255:red, 0; green, 0; blue, 0 }  ,fill opacity=1 ] (174,120) .. controls (174,116.69) and (176.69,114) .. (180,114) .. controls (183.31,114) and (186,116.69) .. (186,120) .. controls (186,123.31) and (183.31,126) .. (180,126) .. controls (176.69,126) and (174,123.31) .. (174,120) -- cycle ;
\draw  [fill={rgb, 255:red, 0; green, 0; blue, 0 }  ,fill opacity=1 ] (175,180.75) .. controls (175,177.3) and (177.8,174.5) .. (181.25,174.5) .. controls (184.7,174.5) and (187.5,177.3) .. (187.5,180.75) .. controls (187.5,184.2) and (184.7,187) .. (181.25,187) .. controls (177.8,187) and (175,184.2) .. (175,180.75) -- cycle ;
\draw  [fill={rgb, 255:red, 0; green, 0; blue, 0 }  ,fill opacity=1 ] (355,180.5) .. controls (355,176.91) and (357.91,174) .. (361.5,174) .. controls (365.09,174) and (368,176.91) .. (368,180.5) .. controls (368,184.09) and (365.09,187) .. (361.5,187) .. controls (357.91,187) and (355,184.09) .. (355,180.5) -- cycle ;
\draw  [fill={rgb, 255:red, 0; green, 0; blue, 0 }  ,fill opacity=1 ] (444.75,150) .. controls (444.75,146.55) and (447.55,143.75) .. (451,143.75) .. controls (454.45,143.75) and (457.25,146.55) .. (457.25,150) .. controls (457.25,153.45) and (454.45,156.25) .. (451,156.25) .. controls (447.55,156.25) and (444.75,153.45) .. (444.75,150) -- cycle ;
\draw  [fill={rgb, 255:red, 0; green, 0; blue, 0 }  ,fill opacity=1 ] (384,90.75) .. controls (384,87.3) and (386.8,84.5) .. (390.25,84.5) .. controls (393.7,84.5) and (396.5,87.3) .. (396.5,90.75) .. controls (396.5,94.2) and (393.7,97) .. (390.25,97) .. controls (386.8,97) and (384,94.2) .. (384,90.75) -- cycle ;

\draw (66,139.4) node [anchor=north west][inner sep=0.75pt]    {$c_{2}^{l} =0,\ d_{2}^{l} =2$};
\draw (256,189.4) node [anchor=north west][inner sep=0.75pt]    {$d^{b} =6$};
\draw (84,64.9) node [anchor=north west][inner sep=0.75pt]    {$c_{1}^{l} =-1,\ d_{1}^{l} =2$};
\draw (335,46.4) node [anchor=north west][inner sep=0.75pt]  [font=\small]  {$c_{1}^{r} =3,\ d_{1}^{r} =1$};
\draw (435,102.4) node [anchor=north west][inner sep=0.75pt]    {$c_{2}^{r} =1,\ d_{2}^{r} =2$};
\draw (408,176.4) node [anchor=north west][inner sep=0.75pt]    {$c_{3}^{r} =-3,\ d_{3}^{r} =1$};
\draw (245,29.4) node [anchor=north west][inner sep=0.75pt]    {$d^{t} =2$};

\end{tikzpicture}
\caption{The $h$-transverse polygon associated to the data in \cref{example1}.}
\label{fig-htransverse}
\end{figure}
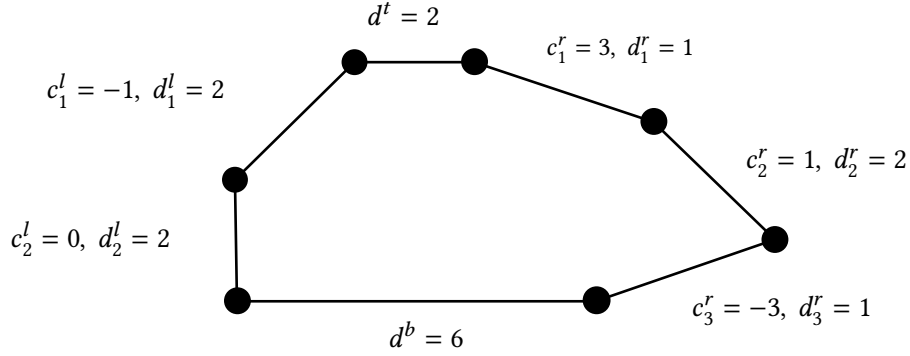
\end{example}

Let $P(\textbf{c},\textbf{d})$ an $h$-transverse polygon as above. Furthermore, consider $g\ge0$ and four sequences $\alpha,\beta,\tilde{\alpha},\tilde{\beta}$ with

\begin{equation}
    \sum_{i}i(\alpha_i+\beta_i)=d^b\quad\textrm{and}\quad\sum_{i}i(\tilde{\alpha}_i+\tilde{\beta}_i)=d^t.
\end{equation}

Let $D_b$ and $D_t$ the toric divisors corresponding to the bottom and top edge of $P(\textbf{c},\textbf{d})$ respectively. The enumerative problem, we study in this article counts genus $g$ curves in $S(\textbf{c})$ passing through an appropriate number of points and satisfying tangency conditions at the toric divisors $D_b$ and $D_t$. The tangency conditions at $D_b$ are determined by $\alpha$ and $\beta$, while $\tilde\alpha$ and $\tilde{\beta}$ determine the conditions at $D_t$.  We define $a=\displaystyle\sum_{i=1}^nd_i^r=\sum_{j=1}^md_j^l$ and $l=2a+g+\sum_{i\ge1}(\beta_i+\tilde{\beta}_i)-1$. Finally, we say that a curve in $S(\textbf{c})$ with Newton polygon $P(\textbf{c},\textbf{d})$ has multidegree $\textbf{d}$.

Now, we are ready to define double Gromov--Witten invariants for $h$-transverse polygons.

\begin{definition}
\label{def-count}
In the set-up above, we consider a generic configuration
\begin{equation}
    \omega=\bigcup_{i\geq0}\left((q_j^i)_{j\leq\alpha_i}\cup(\tilde q_j^i)_{j\leq\tilde\alpha_i}\right)\cup(p_j)_{j\leq l},
\end{equation}
where $q_j^i\in D_b,\tilde{q}_j^i\in D_t$ and $p_i\in S(\textbf{c})\backslash(D_t\cup D_b)$.

We denote by $N_{\textbf c,g}^{\alpha,\beta,\tilde\alpha,\tilde\beta}(\textbf{d})$ the number of irreducible complex algebraic  curves $C$ of multidegree $\textbf{d}$ in $S(\textbf{c})$ of genus $g$ such that:

\begin{itemize}
    \item $C$ passes through all the points $q^i_j$, $\tilde q^i_j$ and $p_i$;
    \item $C$ has order of contact $i$ with $D_b$ at $q^i_j$, and has $\beta_i$ other non-prescribed points with order of contact $i$ with $D_b$;
    \item $C$ has order of contact $i$ with $D_t$ at $\tilde q^i_j$, and has $\tilde\beta_i$ other non-prescribed points with order of contact $i$ with $D_t$.
\end{itemize}

This number is finite and independent of the chosen generic configuration of points and is referred to as the \textit{double Gromov-Witten invariant} of $S(\textbf{c})$. When we allow reducible curves, we denote the resulting invariant by $N_{\textbf c,g}^{\alpha,\beta,\tilde\alpha,\tilde\beta,\bullet}(\textbf{d})$.
\end{definition}

\subsection{Results}\label{res}
Our main objective is to study the polynomiality properties of the invariant $N_{\textbf c,g}^{\alpha,\beta,\tilde\alpha,\tilde\beta}(\textbf{d})$. For this, we first reparametrise the invariant and interpret it as a function.

Let $\textbf{d}^l,\textbf{d}^r$ as above and $n_1,n_2\ge0$. Then, we define
\begin{equation}
\Lambda=\left\{(x_1,\dots,x_{n_1},y_1,\dots,y_{n_2})\in\mathbb{Z}^{n_1}\times\mathbb{Z}^{n_2}\bigg|\sum_{i=1}^{n_1}x_i+\sum_{j=1}^{n_2}y_j+\sum_{i=1}^nc_i^rd_i^r-\sum_{j=1}^mc_j^ld_j^l=0\right\}.
\end{equation}
For $(\textbf{x},\textbf{y})\in\Lambda$, we associate a tuple $(\alpha,\beta,\tilde{\alpha},\tilde{\beta})$ as follows: \begin{multicols}{2}
\begin{itemize}
    \item $\alpha_i$ is the number of elements $x_j=-i$;
    \item $\beta_i$ is the number of elements $y_j=-i$;
    \item $\tilde\alpha_i$ is the number of elements $x_j=i$;
    \item $\tilde\beta_i$ is the number of elements $y_j=i$.
\end{itemize}
\end{multicols} 
Obviously $(\textbf{x},\textbf{y})$ and $(\alpha,\beta,\tilde{\alpha},\tilde{\beta})$ determine each other up to a permutation of the entries of $(\textbf{x},\textbf{y})$. In what follows, we consider two maps

\begin{align}\label{eq3}
    F_{(\textbf{d}^r;\textbf{d}^l),\textbf c,g}^{n_1,n_2}:&\quad\Lambda\quad\longrightarrow\quad\mathbb Z\qquad\qquad\qquad F_{(\textbf{d}^r;\textbf{d}^l),g}^{n_1,n_2}:\quad\Lambda\times\mathbb Z^{n+m}\quad\longrightarrow\quad\mathbb Z\\
    &(\textbf{x},\textbf{y})\>\>\longmapsto\quad N_{\textbf c,g}^{\alpha,\beta,\tilde\alpha,\tilde\beta}(\textbf d)\qquad\qquad\qquad\quad(\textbf{x},\textbf{y},\textbf{c})\>\>\longmapsto\quad N_{\textbf c,g}^{\alpha,\beta,\tilde\alpha,\tilde\beta}(\textbf d).
\end{align}

In order to study the polynomiality properties of $ F_{(\textbf{d}^r;\textbf{d}^l),\textbf c,g}^{n_1,n_2}$, consider the hyperplane arrangement in $\Lambda$ consisting of all hyperplanes

\begin{equation}
        \sum_{i\in S}x_i+\sum_{j\in T}y_j+\sum_{i=1}^nc_i^rk_i-\sum_{j=1}^mc_j^lt_j=0
    \end{equation}
    \begin{equation}
        y_i-y_j=0\qquad 1\leq i<j\leq n_2
\end{equation}
where $S\subseteq[n_1]$, $T\subseteq[n_2]$, $0\leq k_i\leq d_i^r$ for $i=1,\dots,n$ and $0\leq t_j\leq d_j^l$ for $j=1,\dots,m$. We denote this hyperplane arrangement by $\mathcal{H}^{n_1,n_2}(\textbf{c})$. We call the elements of $\mathcal{H}^{n_1,n_2}(\textbf{c})$ walls and the connected components of $\Lambda\backslash\mathcal{H}^{n_1,n_2}(\textbf{c})$ \textit{chambers}. Furthermore, we define $\tilde{\mathcal{H}}^{n_1,n_2}\subset\Lambda\times\mathbb Z^{n+m}$ to be the intersection of $\mathcal{H}^{n_1,n_2}(\textbf{c})$ and $\{\textbf{c}=(\textbf{c}^r,\textbf{c}^l)\in\mathbb Z^{n+m}|c_1^r>\dots>c_n^r,\>c_1^l<\dots<c_m^l\}$. $\tilde{\mathcal{H}}^{n_1,n_2}$ is a hyperplane arrangement in $\Lambda\times\mathbb Z^{n+m}$.

The following is the main result of this work.

\begin{theorem}\label{thm1}
    Let $(\textbf d^r;\textbf d^l)$ be a vector with positive integer coordinates and $g\geq0,n_1,n_2>0$ fixed integers and $\textbf c=(\textbf c^r;\textbf c^l)\in\mathbb Z^{n+m}$ such that $c_1^r>\dots>c_n^r$ and $c_1^l<\dots<c_m^l$. The function $F_{(\textbf d^r;\textbf d^l),\textbf c,g}^{n_1,n_2}(\textbf x,\textbf y)$ of double Gromov-Witten invariants of the toric surface $S(\textbf c)$ is polynomial in each chamber of $\Lambda\backslash\mathcal{H}^{n_1,n_2}(\textbf{c})$. Furthermore, if we let the vector $\textbf{c}\in\mathbb Z^{n+m}$ vary, then the function $F_{(\textbf d^r;\textbf d^l),g}^{n_1,n_2}(\textbf x,\textbf y,\textbf c)$ depending also on the toric surface $S(\textbf{c})$ is polynomial in each chamber of $\tilde{\mathcal{H}}^{n_1,n_2}$.
\end{theorem}

As a direct application of the same methods of the proof of \cite[Theorem 1.4]{ardila2017double}, we note also the following result.

\begin{theorem}\label{thm2}
    Each polynomial piece of $F_{(\textbf{d}^r;\textbf{d}^l),\textbf c,g}^{n_1,n_2}(\textbf x,\textbf y)$ has degree $n_2+3g+2a-2$, and is either even or odd.
\end{theorem}

\begin{remark}
    For the special case of $S(\textbf{c})$ a Hirzebruch surface, \cref{thm1} was proved in \cite[Theorem 1.3]{ardila2017double} and \cref{thm2} was proved in \cite[Theorem 1.4]{ardila2017double}.
\end{remark}

The results in \cref{thm1,thm2} establish the piecewise polynomiality of $N_{\textbf c,g}^{\alpha,\beta,\tilde\alpha,\tilde\beta}(\textbf d)$ and determine the parity of the involved polynomials for a fixed toric surface. In particular, the second part of \cref{thm1} shows that $N_{\textbf c,g}^{\alpha,\beta,\tilde\alpha,\tilde\beta}(\textbf d)$ also behaves polynomially while changing the underlying toric surface, answering the aforementioned question posed in \cite{ardila2017double} affirmatively. This motivates our choice of the word \textit{universal} in the title: the piecewise polynomiality of these invariants persists while changing across a family of toric surfaces.

Finally, in \cref{sec-fockspace}, we explore the relation between the invariants $N_{\textbf c,g}^{\alpha,\beta,\tilde\alpha,\tilde\beta}(\textbf d)$ and the bosonic Fock space. The bosonic Fock space is a vector space with a basis indexed by tuples of partitions. The bosonic Fock space carries a natural inner product. The two--dimensional Heisenberg algebra is a family of operators acting on the bosonic Fock space. Employing the inner product, we may associate to an operator in the Heisenberg algebra a \textbf{vaccuum expectation}. Building on previous work, linking Floor diagrams to this formalism (see e.g. \cite{block2016fock,cavalieri2021counting}), we construct operators whose vaccuum expectations are equal to our invariants $N_{\textbf c,g}^{\alpha,\beta,\tilde\alpha,\tilde\beta}(\textbf d)$. The precise statement may be found in \cref{thm:fockspace}.

\section{Floor diagrams}\label{sec-floordiag}

We introduce floor diagrams following the notation in \cite{ardila2013universal} and \cite{ardila2017double}. Let $\textbf{d}=(d^t;\textbf{d}^r;\textbf{d}^l)$ be a vector with positive integer coordinates and $\textbf{c}=(\textbf{c}^r;\textbf{c}^l)\in\mathbb Z^n\times\mathbb Z^m$ such that $c_1^r>\dots>c_n^r$ and $c_1^l<\dots<c_m^l$. Let $P(\textbf{c},\textbf{d})$ be a $h$-transverse polygon and $S(\textbf{c})$ the corresponding toric surface. We denote by $D_r$ and $D_l$ two multisets containing the directions of the right and left sides respectively:

\begin{equation}
    D_r=\{\underbrace{c_1^r,\dots,c_1^r}_{d_1^r-times},\dots,\underbrace{c_n^r,\dots,c_n^r}_{d_n^r-times}\}\qquad D_l=\{\underbrace{c_1^l,\dots,c_1^l}_{d_1^l-times},\dots,\underbrace{c_m^l,\dots,c_m^l}_{d_m^l-times}\}.
\end{equation}
Note that $|D_r|=|D_l|=a$. Let $r=(r_1,\dots,r_a)$ and $l=(l_1,\dots,l_a)$ be permutations of $D_r$ and $D_l$ respectively.

\begin{definition}\label{def-floordiag}
    A \textbf{marked floor diagram} $\mathcal D$ \textbf{of multidegree d} for $S(\textbf{c})$ is a tuple $(V,E,w)$, such that
    \begin{enumerate}
        \item The vertex set $V$ is decomposed as $V=L\cup C\cup R$. Moreover, we have that $C$ is totally ordered from left to right, while $L=\{\tilde q_1,\dots,\tilde q_{\tilde k}\}$ is unordered and to the left of $C$, and $R=\{q_1,\dots,q_k\}$ is unordered and to the right of $C$.
        \item The vertices in $V$ are coloured black, white and grey. {Moreover, every vertex in $L$ and $R$ is white and there are $a$ black vertices in $C$.}
        \item The set of edges $E$ is directed from left to right, such that:
        \begin{itemize}
            \item the resulting graph is connected;
            \item every white vertex has valency one and is connected to precisely one black vertex;
            \item every grey vertex has valency two with one incoming and one outgoing edge, each connecting it to a black vertex;
            \item {there is no black-black edge.}
        \end{itemize}
        \item We have a map
        \begin{equation}
            w\colon E\to\mathbb{Z}_{>0}
        \end{equation}
       such that if we define the divergence of $v$ to be 
        $$\text{div}(v)=\sum_{e:v\to v'}w(e)-\sum_{e:v'\to v}w(e)$$ 
        then
        \begin{itemize}
            \item $\text{div}(B_i)=r_i-l_i$, where $B_i$ is the $i$-th black vertex in $\mathcal D$, for all $i=1,\dots,a$;
            \item $\text{div}(v)=0$ for every grey vertex $v$.
        \end{itemize}
        \item {We define the \textbf{divergence multiplicity vector} to be the vector of sequences $(\alpha,\beta,\tilde\alpha,\tilde\beta)$ such that:
        \begin{multicols}{2}
        \begin{itemize}
            \item $\alpha_i$ is the number of vertices $v\in R$ such that $\text{div}(v)=-i$; 
            \item $\beta_i$ is the number of white vertices $v\in C$ such that $\text{div}(v)=-i$;
            \item $\tilde\alpha_i$ is the number of vertices $v\in L$ such that $\text{div}(v)=i$;
            \item $\tilde\beta_i$ is the number of white vertices $v\in C$ such that $\text{div}(v)=i$;
        \end{itemize}
        \end{multicols}
        Then the multidegree \textbf{d} satisfies the following equation
        \begin{equation}            
            \sum_ii(\alpha_i+\beta_i)=d^t+\sum_{i=1}^nd_i^rc_i^r-\sum_{j=1}^md_j^lc_j^l\quad\text{and}\quad\sum_ii(\tilde\alpha_i+\tilde\beta_i)=d^t.
        \end{equation}}
    \end{enumerate}
\end{definition}

{We associate to a marked floor diagram $\mathcal{D}$ the following quantities:}

\begin{itemize}
    \item {the \textbf{type} of $\mathcal{D}$ is the pair $(n_1,n_2)$, where $n_1$ is the number of white vertices in $L\cup R$ and $n_2$ is the number of white vertices in $C$.
    \item The \textbf{divergence sequence} of $\mathcal D$ is a vector $(\textbf{x},\textbf{y})\in\mathbb Z^{n_1}\times\mathbb Z^{n_2}$ of length $n_1+n_2$ defined as follows  
    \begin{itemize}
        \item $\textbf{x}=(\text{div}(\tilde q_1),\dots,\text{div}(\tilde q_{\tilde k}),\text{div}(q_1),\dots,\text{div}(q_k))$ the sequence of divergences of white vertices in $L$ and $R$. We will often refer to \textbf{x} as left-right sequence;
        \item $\textbf{y}$ the sequence of divergences of white vertices in $C$, listed from left to right.
    \end{itemize}
    Note that the vector $(\textbf{x},\textbf{y})$ satisfies the following equation:
    \begin{equation}
    \sum_{i=1}^{n_1}x_i+\sum_{j=1}^{n_2}y_j=\sum_{j=1}^mc_j^ld_j^l-\sum_{i=1}^nc_i^rd_i^r=d^t-d^b
    \end{equation}
    since the divergences of $\mathcal D$ sum to zero. Furthermore, $(\textbf{x},\textbf{y})$ determines\\
    $(\alpha,\beta,\tilde\alpha,\tilde\beta)=(\alpha(\textbf{x}),\beta(\textbf{y}),\tilde\alpha(\textbf{x}),\tilde\beta(\textbf{y}))$ and viceversa (up to a permutation of the entries of the vector $(\textbf{x},\textbf{y})$).}
    \item We define the \textbf{genus} $g(\mathcal D)$ of $\mathcal D$ to be the first Betti number of the underlying graph, i.e. $g(\mathcal D)=1-|V|+|E|$. {By definition, the genus of $\mathcal D$ can be rewritten as $g(\mathcal D)=1-a+v_g$, where $v_g$ is the number of grey vertices.}
    \item An edge is \textbf{internal} if it connects two vertices of $C$. We define the \textbf{multiplicity} $\mu(\mathcal D)$ of $\mathcal{D}$ as
    \begin{equation}
        \mu(\mathcal{D})=\prod w(e)
    \end{equation}
    where the product runs over all internal edges. 
\end{itemize}

\begin{remark}
    {It is worth emphasising that, in this setup, the divergence of the black vertices of a given marked floor diagram depends on the permutations $r$ and $l$, whereas in the case treated by Ardila and Brugallé in \cite{ardila2017double}, it is constant.} 
\end{remark}

\begin{remark}\label{rmk-internal_edges}
    {Note that from the definition immediately follows that a marked floor diagram $\mathcal{D}$ of multidegree $\textbf{d}$ for $S(\textbf{c})$ of type $(n_1,n_2)$ has $n_2+2(g+a-1)$ internal edges.}
\end{remark}

{From this point forward, we will refer to \textit{marked floor diagrams} as \textit{floor diagrams} for convinience. Now, we state the correspondence theorem for our enumerative problem.}

\begin{theorem}\label{thm4}
    Let $\textbf d=(d^t;\textbf d^r;\textbf d^l)$ be {a vector of positive integer numbers}, $g\geq0$ an integer and \textbf{x} a vector with coordinates in $\mathbb{Z}\setminus\{0\}$. We write $\alpha(\textbf x)=\alpha$ and $\tilde\alpha(\textbf x)=\tilde\alpha$. Then, for any two sequences of non-negative integer numbers $\beta=(\beta_i)_{i\geq1}$ and $\tilde\beta=(\tilde\beta_i)_{i\geq1}$ such that
    \begin{equation}
        \sum_ii(\alpha_i+\beta_i)=d^t+\sum_{i=1}^nc_i^rd_i^r-\sum_{j=1}^mc_j^ld_j^l\quad\text{and}\quad\sum_ii(\tilde\alpha_i+\tilde\beta_i)=d^t,
    \end{equation}
    one has
    \begin{equation}
        N_{\textbf c,g}^{\alpha,\beta,\tilde\alpha,\tilde\beta}(\textbf d)=\sum_{\mathcal D}\mu(\mathcal D)
    \end{equation}
    where the sum runs over all floor diagrams $\mathcal D$ of multidegree \textbf d, genus $g$, left-right sequence \textbf x, and divergence multiplicity vector $(\alpha,\beta,\tilde\alpha,\tilde\beta)$ for $S(\textbf c)$.
\end{theorem}

\begin{proof}
   In \cite{brugalle2008floor} a proof of this result is given for

    \begin{equation}
        \alpha=\tilde\alpha=0,\>\>\beta=(d^b,0,\dots,0),\>\>\tilde\beta=(d^t,0,\dots,0).
    \end{equation}

    employing Mikhalkin's correspondence Theorem \cite{mikhalkin2005enumerative} that expresses enumerations of curves in toric surfaces via \textit{tropical curves} which are piecewise linear graphs in the plane. In \cite[Theorem 2]{shustin2012tropical} a generalisation of Mikhalkin's correspondence Theorem which covers the case of curves satisfying tangency conditions with toric divisors was proved.  For the sake of convinience, we provide a sketch of the proof for a tropically inclined reader. Shustin's strategy was to dissipate each point $p$ with multiplicity $k>1$ in $k$ points in a neighbourhood of $p$. Considering the curves passing through the new configuration of points and then specialising back to the original one, Shustin proved that each of these curves converges to a curve satisfying the prescibed tangency. This yields a correspondence theorem between enumeration of classical and tropical curves with point and tangency conditions. In order to obtain our result, we need to construct a bijection between tropical curves and floor diagrams. Let $P=P(\textbf{c},\textbf{d})$ be an $h$-transverse polygon and $S(\textbf{c})$ be the corresponding toric surface. There is a correspondence between lattice subdivisions of $P$ and the tropical curves of multidegree $\textbf{d}$ having $P$ as a Newton polygon:
    
    \begin{itemize}
        \item the vertices of the tropical curve correspond to polygons in the lattice subdivision;
        \item the edges of the tropical curve emanated from a vertex $v$ correspond to the normal lines to the sides of the polygon corresponding to $v$;
        \item the faces determined by two edges emanated from a vertex $v$ correspond to vertices of the polygon in the lattice subdivision corresponding to $v$. 
    \end{itemize}

    Fixing a preferred direction in $\mathbb R^2$ allows to distinguish between edges of the tropical curve which are parallel to this direction and the others. The direction we fix is given by the vector $(0,1)$. We call the unbounded edges of the tropical curve leaves and the edges in direction $(0,1)$ elevators. A connected component of the tropical curve after removing bounded and unbounded elevators is called a floor.\\
    We want to associate a decorated graph $\mathcal D(T)$ to the tropical curve $T$: black vertices of $\mathcal D(T)$ correspond to floors of $T$, grey vertices of $\mathcal D(T)$ correspond to bounded elevators of $T$ and white vertices of $\mathcal D(T)$ correspond to unbounded elevators of $T$. More precisely, each vertex of the graph corresponds to a point in a fixed generic configuration. The weights of black-grey edges are the weights of the corresponding elevators, while the weights of the black-white edges are the weights of the corresponding leaves. We divide the graph in three blocks by adding vertical dashed lines, the white vertices in the central block correspond to non-prescribed points, while the white vertices in the left and right blocks correspond to points in the generic configuration belonging to $D_t$ and $D_b$ respectively. Furthermore, to each vertex we associate a number called divergence:

    \begin{itemize}
        \item to each black vertex we associate the difference between the slopes of the right and left leaves of the corresponding floor;
        \item to each grey vertex we associate $0$;
        \item to each white vertex we associate the weight of the corresponding leaf if the leaf points to $+\infty$ or minus the weight of the corresponding leaf if the leaf points to $-\infty$.
    \end{itemize}

    The construction explained above yields a bijection between tropical curves and floor diagrams.
\end{proof}

\begin{example}
    In \cref{fig-floorex1}, we give an example of a floor diagram constructed by using the technique explained in the proof of \cref{thm4} in the case $P=P(\textbf c,\textbf d)$, where $\textbf c=(2;-1,0)$, $\textbf d=(1;2;1,1)$ and $x_1=1$, $x_2=-5$ and $y_1=-1$.
\end{example}

\begin{figure}
   \tikzset{every picture/.style={line width=0.75pt}} 

\begin{tikzpicture}[x=0.75pt,y=0.75pt,yscale=-0.9,xscale=0.9]

\draw    (209.73,109.09) -- (253.81,110) ;
\draw [shift={(253.81,110)}, rotate = 1.18] [color={rgb, 255:red, 0; green, 0; blue, 0 }  ][fill={rgb, 255:red, 0; green, 0; blue, 0 }  ][line width=0.75]      (0, 0) circle [x radius= 5.36, y radius= 5.36]   ;
\draw [shift={(205.37,109)}, rotate = 1.18] [color={rgb, 255:red, 0; green, 0; blue, 0 }  ][line width=0.75]      (0, 0) circle [x radius= 5.36, y radius= 5.36]   ;
\draw    (253.81,110) -- (299.52,110) ;
\draw [shift={(303.88,110)}, rotate = 0] [color={rgb, 255:red, 0; green, 0; blue, 0 }  ][line width=0.75]      (0, 0) circle [x radius= 5.36, y radius= 5.36]   ;
\draw [shift={(253.81,110)}, rotate = 0] [color={rgb, 255:red, 0; green, 0; blue, 0 }  ][fill={rgb, 255:red, 0; green, 0; blue, 0 }  ][line width=0.75]      (0, 0) circle [x radius= 5.36, y radius= 5.36]   ;
\draw    (308.24,110) -- (353.96,110) ;
\draw [shift={(353.96,110)}, rotate = 0] [color={rgb, 255:red, 0; green, 0; blue, 0 }  ][fill={rgb, 255:red, 0; green, 0; blue, 0 }  ][line width=0.75]      (0, 0) circle [x radius= 5.36, y radius= 5.36]   ;
\draw [shift={(303.88,110)}, rotate = 0] [color={rgb, 255:red, 0; green, 0; blue, 0 }  ][line width=0.75]      (0, 0) circle [x radius= 5.36, y radius= 5.36]   ;
\draw    (353.96,110) -- (398.03,110) ;
\draw [shift={(402.39,110)}, rotate = 0] [color={rgb, 255:red, 0; green, 0; blue, 0 }  ][line width=0.75]      (0, 0) circle [x radius= 5.36, y radius= 5.36]   ;
\draw [shift={(353.96,110)}, rotate = 0] [color={rgb, 255:red, 0; green, 0; blue, 0 }  ][fill={rgb, 255:red, 0; green, 0; blue, 0 }  ][line width=0.75]      (0, 0) circle [x radius= 5.36, y radius= 5.36]   ;
\draw    (353.96,110) .. controls (385.64,81.05) and (423.44,80.04) .. (447.44,106.96) ;
\draw [shift={(450,110)}, rotate = 51.57] [color={rgb, 255:red, 0; green, 0; blue, 0 }  ][line width=0.75]      (0, 0) circle [x radius= 5.36, y radius= 5.36]   ;
\draw [shift={(353.96,110)}, rotate = 317.58] [color={rgb, 255:red, 0; green, 0; blue, 0 }  ][fill={rgb, 255:red, 0; green, 0; blue, 0 }  ][line width=0.75]      (0, 0) circle [x radius= 5.36, y radius= 5.36]   ;
\draw  [dash pattern={on 4.5pt off 4.5pt}]  (222.61,77) -- (221.79,146) ;
\draw  [dash pattern={on 4.5pt off 4.5pt}]  (414.7,74) -- (414.7,147) ;
\draw  [fill={rgb, 255:red, 155; green, 155; blue, 155 }  ,fill opacity=1 ] (299.37,110) .. controls (299.37,106.96) and (301.39,104.5) .. (303.88,104.5) .. controls (306.37,104.5) and (308.4,106.96) .. (308.4,110) .. controls (308.4,113.04) and (306.37,115.5) .. (303.88,115.5) .. controls (301.39,115.5) and (299.37,113.04) .. (299.37,110) -- cycle ;
\draw [line width=2.25]    (34.63,90) -- (59.25,90) ;
\draw [line width=2.25]    (34.63,90) -- (10,120) ;
\draw [line width=2.25]    (10,120) -- (10,149) ;
\draw [line width=2.25]    (10,149) -- (158.58,149) ;
\draw [line width=2.25]    (59.25,90) -- (158.58,149) ;
\draw    (10,120) -- (108.92,119.5) ;
\draw    (10,120) -- (59.25,90) ;
\draw    (10,149) -- (108.92,119.5) ;
\draw    (108.92,119.5) -- (133.96,149) ;
\draw [line width=3]    (483,166) -- (553.33,166.76) ;
\draw [line width=3]    (520,67) -- (553.33,126.17) ;
\draw [line width=3]    (476.63,27.01) -- (520,67) ;
\draw    (520,18.22) -- (520,67) ;
\draw [line width=3]    (553.33,126.17) -- (597.43,44.65) ;
\draw [line width=3]    (553.33,166.76) -- (564.27,208.37) ;
\draw [line width=3]    (564.27,208.37) -- (586.5,188.25) ;
\draw [line width=3]    (586.5,188.25) -- (608,147.32) ;
\draw    (564.27,208.37) -- (565,278.98) ;
\draw    (586.5,188.25) -- (586.86,280) ;
\draw    (553.33,126.17) -- (553.33,166.76) ;

\draw (228.93,85.4) node [anchor=north west][inner sep=0.75pt]    {$1$};
\draw (390.64,65.4) node [anchor=north west][inner sep=0.75pt]    {$5$};
\draw (373.4,116.4) node [anchor=north west][inner sep=0.75pt]    {$1$};
\draw (551,238.4) node [anchor=north west][inner sep=0.75pt]    {$5$};
\draw (274,88.4) node [anchor=north west][inner sep=0.75pt]    {$4$};
\draw (324,89.4) node [anchor=north west][inner sep=0.75pt]    {$4$};
\draw (557,135.4) node [anchor=north west][inner sep=0.75pt]    {$4$};

\end{tikzpicture}

\caption{}
\label{fig-floorex1}
\end{figure}

\begin{example}\label{example3}
    Let us consider $n=2$, $m=3$, $n_1=2$, $n_2=1$, $g=0$, $\textbf c^r=(c_1^r,c_2^r)$, $\textbf c^l=(c_1^l,c_2^l,c_3^l)$, $\textbf d^r=(2,2)$ and $\textbf d^l=(1,1,2)$. Then, $a=d_1^r+d_2^r=d_1^l+d_2^l+d_3^l=4$ and the multisets are 
    
    \begin{equation}
    D_r=\{c_1^r,c_1^r,c_2^r,c_2^r\}\qquad D_l=\{c_1^l,c_2^l,c_3^l,c_3^l\}.
    \end{equation}
    Let $r=(c_1^r,c_2^r,c_1^r,c_2^r)$ and $l=(c_3^l,c_1^l,c_3^l,c_2^l)$, so $r-l=(c_1^r-c_3^l,c_2^r-c_1^l,c_1^r-c_3^l,c_2^r-c_2^l)$.  The floor diagram in \cref{fig-floorex2} is of type $(2,1)$, with divergence sequence $(3,-4,-2)$, divergence multiplicity vector $(\alpha,\beta,\tilde\alpha,\tilde\beta)=(0001,01,001,0)$, multidegree $\textbf d=(3;2,2;1,1,2)$ and multiplicity 
    
    \begin{equation}
    \mu(\mathcal D)=2(c_{1}^{r} -c_{3}^{l} +3)^2(c_{2}^{r} -c_{1}^{l} +c_{1}^{r} -c_{3}^{l} +3)^2(6-c_{2}^{r}+c_{2}^{l})^2.
    \end{equation}
    This floor diagram contributes non-zero to the Gromov-Witten invariant 
    
    \begin{equation}
    N_{\textbf c,0}^{0001,01,001,0}(3;2,2;1,1,2)=F^{2,1}_{(2,2;1,1,2),\textbf c,0}(3,-4,-2)
    \end{equation}
    of the surface $S(\textbf c)$ in the case all weights of the edges are positive, namely
    \begin{equation}
        \begin{cases}
            c_1^r-c_3^l+3>0\\
            c_2^r-c_1^l+c_1^r-c_3^l+3>0\\
            6-c_2^r+c_2^l>0
        \end{cases}
    \end{equation}
 \end{example}
   
\begin{center} 
\begin{figure}[H]
\tikzset{every picture/.style={line width=0.75pt}} 

\begin{tikzpicture}[x=0.75pt,y=0.75pt,yscale=-1,xscale=1]

\draw  [dash pattern={on 4.5pt off 4.5pt}]  (84.67,65.33) -- (84.67,80.33) -- (84.67,166.33) ;
\draw  [dash pattern={on 4.5pt off 4.5pt}]  (577.33,65.67) -- (577.33,167.67) ;
\draw   (48.67,108) .. controls (48.67,104.13) and (51.8,101) .. (55.67,101) .. controls (59.53,101) and (62.67,104.13) .. (62.67,108) .. controls (62.67,111.87) and (59.53,115) .. (55.67,115) .. controls (51.8,115) and (48.67,111.87) .. (48.67,108) -- cycle ;
\draw  [fill={rgb, 255:red, 0; green, 0; blue, 0 }  ,fill opacity=1 ] (106.33,108.33) .. controls (106.33,104.47) and (109.47,101.33) .. (113.33,101.33) .. controls (117.2,101.33) and (120.33,104.47) .. (120.33,108.33) .. controls (120.33,112.2) and (117.2,115.33) .. (113.33,115.33) .. controls (109.47,115.33) and (106.33,112.2) .. (106.33,108.33) -- cycle ;
\draw  [fill={rgb, 255:red, 128; green, 128; blue, 128 }  ,fill opacity=1 ] (166.67,109.33) .. controls (166.67,105.47) and (169.8,102.33) .. (173.67,102.33) .. controls (177.53,102.33) and (180.67,105.47) .. (180.67,109.33) .. controls (180.67,113.2) and (177.53,116.33) .. (173.67,116.33) .. controls (169.8,116.33) and (166.67,113.2) .. (166.67,109.33) -- cycle ;
\draw  [fill={rgb, 255:red, 0; green, 0; blue, 0 }  ,fill opacity=1 ] (226,110) .. controls (226,106.13) and (229.13,103) .. (233,103) .. controls (236.87,103) and (240,106.13) .. (240,110) .. controls (240,113.87) and (236.87,117) .. (233,117) .. controls (229.13,117) and (226,113.87) .. (226,110) -- cycle ;
\draw  [fill={rgb, 255:red, 128; green, 128; blue, 128 }  ,fill opacity=1 ] (285.33,110) .. controls (285.33,106.13) and (288.47,103) .. (292.33,103) .. controls (296.2,103) and (299.33,106.13) .. (299.33,110) .. controls (299.33,113.87) and (296.2,117) .. (292.33,117) .. controls (288.47,117) and (285.33,113.87) .. (285.33,110) -- cycle ;
\draw  [fill={rgb, 255:red, 0; green, 0; blue, 0 }  ,fill opacity=1 ] (345.33,110) .. controls (345.33,106.13) and (348.47,103) .. (352.33,103) .. controls (356.2,103) and (359.33,106.13) .. (359.33,110) .. controls (359.33,113.87) and (356.2,117) .. (352.33,117) .. controls (348.47,117) and (345.33,113.87) .. (345.33,110) -- cycle ;
\draw  [fill={rgb, 255:red, 128; green, 128; blue, 128 }  ,fill opacity=1 ] (406.67,110) .. controls (406.67,106.13) and (409.8,103) .. (413.67,103) .. controls (417.53,103) and (420.67,106.13) .. (420.67,110) .. controls (420.67,113.87) and (417.53,117) .. (413.67,117) .. controls (409.8,117) and (406.67,113.87) .. (406.67,110) -- cycle ;
\draw  [fill={rgb, 255:red, 0; green, 0; blue, 0 }  ,fill opacity=1 ] (466,108.67) .. controls (466,104.8) and (469.13,101.67) .. (473,101.67) .. controls (476.87,101.67) and (480,104.8) .. (480,108.67) .. controls (480,112.53) and (476.87,115.67) .. (473,115.67) .. controls (469.13,115.67) and (466,112.53) .. (466,108.67) -- cycle ;
\draw   (526,108) .. controls (526,104.13) and (529.13,101) .. (533,101) .. controls (536.87,101) and (540,104.13) .. (540,108) .. controls (540,111.87) and (536.87,115) .. (533,115) .. controls (529.13,115) and (526,111.87) .. (526,108) -- cycle ;
\draw   (600.67,109.33) .. controls (600.67,105.47) and (603.8,102.33) .. (607.67,102.33) .. controls (611.53,102.33) and (614.67,105.47) .. (614.67,109.33) .. controls (614.67,113.2) and (611.53,116.33) .. (607.67,116.33) .. controls (603.8,116.33) and (600.67,113.2) .. (600.67,109.33) -- cycle ;
\draw    (62.67,108) -- (106.33,108.33) ;
\draw    (120.33,108.33) -- (166.67,109.33) ;
\draw    (180.67,109.33) -- (226,110) ;
\draw    (240,110) -- (285.33,110) ;
\draw    (299.33,110) -- (345.33,110) ;
\draw    (359.33,110) -- (406.67,110) ;
\draw    (420.67,110) -- (466,108.67) ;
\draw    (480,108.67) -- (526,108) ;
\draw    (480,108.67) .. controls (520,78.67) and (561.67,77.67) .. (600.67,109.33) ;

\draw (87.33,93.4) node [anchor=north west][inner sep=0.75pt]  [font=\scriptsize]  {$x_{1}$};
\draw (120,117.73) node [anchor=north west][inner sep=0.75pt]  [font=\tiny]  {$c_{1}^{r} -c_{3}^{l} +x_{1}$};
\draw (180,92.4) node [anchor=north west][inner sep=0.75pt]  [font=\tiny]  {$c_{1}^{r} -c_{3}^{l} +x_{1}$};
\draw (283.67,89.07) node [anchor=north west][inner sep=0.75pt]  [font=\tiny]  {$c_{2}^{r} -c_{1}^{l} +c_{1}^{r} -c_{3}^{l} +x_{1}$};
\draw (406.67,87.4) node [anchor=north west][inner sep=0.75pt]  [font=\tiny]  {$-x_{2} -y_{1} -c_{2}^{r} +c_{2}^{l}$};
\draw (501.67,117.73) node [anchor=north west][inner sep=0.75pt]  [font=\tiny]  {$-y_{1}$};
\draw (529.33,73.73) node [anchor=north west][inner sep=0.75pt]  [font=\tiny]  {$-x_{2}$};
\draw (358.67,125.73) node [anchor=north west][inner sep=0.75pt]  [font=\tiny]  {$-x_{2} -y_{1} -c_{2}^{r} +c_{2}^{l}$};
\draw (224.33,122.4) node [anchor=north west][inner sep=0.75pt]  [font=\tiny]  {$c_{2}^{r} -c_{1}^{l} +c_{1}^{r} -c_{3}^{l} +x_{1}$};

\end{tikzpicture}

\caption{}
\label{fig-floorex2}
\end{figure}
\end{center}

\begin{example}\label{example4}
    Fix $n=2$, $m=3$, $n_1=3$, $n_2=1$, $g=1$, $\textbf c=(\textbf c^r;\textbf c^l)=(c_1^r,c_2^r;c_1^l,c_2^l,c_3^l)$ and $(\textbf d^r,\textbf d^l)=(2,1;1,1,1)$. Furthermore, let $r=(c_1^r,c_2^r,c_1^r)$ and $l=(c_3^l,c_2^l,c_1^l)$, then $r-l=(c_1^r-c_3^l,c_2^r-c_2^l,c_1^r-c_1^l)$. Consider the floor diagram in \cref{fig-floorex3}.

    Note that the weights are uniquely determined by the variable $w$. Since every edge must have positive weight, we get the following inequalities:
    \begin{equation}
        x_1,x_2>0,\qquad y_1,x_3<0,
    \end{equation}
    \begin{equation}
        w>0,\qquad c_2^l-c_2^r+w>0,\qquad -x_3-w-c_1^r+c_1^l>0.
    \end{equation}
    If we assume $c_1^l-c_1^r-x_3>c_2^r-c_2^l$, {then the possible values of $w$ so that we have valid weights are $\max\{0,c_2^r-c_2^l\}\leq w\leq c_1^l-c_1^r-x_3$, and the sum of the multiplicities of the corresponding diagrams is:}
    \begin{equation}
        \sum_{w=\max\{0,c_2^r-c_2^l\}}^{c_1^l-c_1^r-x_3}(-y_1)w^2(c_2^l-c_2^r+w)^2(-x_3-w-c_1^r+c_1^l)^2.
    \end{equation}

\begin{center}
\begin{figure}
    \tikzset{every picture/.style={line width=0.75pt}} 

\begin{tikzpicture}[x=0.75pt,y=0.75pt,yscale=-1,xscale=1]

\draw [line width=0.75]    (25.41,90.08) .. controls (65.07,90.62) and (80.05,85.91) .. (109,111) ;
\draw [shift={(109,111)}, rotate = 40.91] [color={rgb, 255:red, 0; green, 0; blue, 0 }  ][fill={rgb, 255:red, 0; green, 0; blue, 0 }  ][line width=0.75]      (0, 0) circle [x radius= 5.36, y radius= 5.36]   ;
\draw [shift={(21,90)}, rotate = 1.33] [color={rgb, 255:red, 0; green, 0; blue, 0 }  ][line width=0.75]      (0, 0) circle [x radius= 5.36, y radius= 5.36]   ;
\draw [line width=0.75]    (26.52,132) .. controls (73.57,131.92) and (67.29,130.4) .. (109,111) ;
\draw [shift={(109,111)}, rotate = 335.06] [color={rgb, 255:red, 0; green, 0; blue, 0 }  ][fill={rgb, 255:red, 0; green, 0; blue, 0 }  ][line width=0.75]      (0, 0) circle [x radius= 5.36, y radius= 5.36]   ;
\draw [shift={(22,132)}, rotate = 0] [color={rgb, 255:red, 0; green, 0; blue, 0 }  ][line width=0.75]      (0, 0) circle [x radius= 5.36, y radius= 5.36]   ;
\draw    (109,111) -- (156.64,111) ;
\draw [shift={(161,111)}, rotate = 0] [color={rgb, 255:red, 0; green, 0; blue, 0 }  ][line width=0.75]      (0, 0) circle [x radius= 5.36, y radius= 5.36]   ;
\draw [shift={(109,111)}, rotate = 0] [color={rgb, 255:red, 0; green, 0; blue, 0 }  ][fill={rgb, 255:red, 0; green, 0; blue, 0 }  ][line width=0.75]      (0, 0) circle [x radius= 5.36, y radius= 5.36]   ;
\draw    (109,111) .. controls (147.6,82.05) and (179.68,85.69) .. (206.14,108.45) ;
\draw [shift={(209,111)}, rotate = 42.8] [color={rgb, 255:red, 0; green, 0; blue, 0 }  ][line width=0.75]      (0, 0) circle [x radius= 5.36, y radius= 5.36]   ;
\draw [shift={(109,111)}, rotate = 323.13] [color={rgb, 255:red, 0; green, 0; blue, 0 }  ][fill={rgb, 255:red, 0; green, 0; blue, 0 }  ][line width=0.75]      (0, 0) circle [x radius= 5.36, y radius= 5.36]   ;
\draw    (109,111) .. controls (196.75,181.2) and (229.36,143.03) .. (268.01,113.27) ;
\draw [shift={(271,111)}, rotate = 323.13] [color={rgb, 255:red, 0; green, 0; blue, 0 }  ][line width=0.75]      (0, 0) circle [x radius= 5.36, y radius= 5.36]   ;
\draw [shift={(109,111)}, rotate = 38.66] [color={rgb, 255:red, 0; green, 0; blue, 0 }  ][fill={rgb, 255:red, 0; green, 0; blue, 0 }  ][line width=0.75]      (0, 0) circle [x radius= 5.36, y radius= 5.36]   ;
\draw    (212.66,108.4) .. controls (252.62,81.49) and (310.6,89.66) .. (330,111) ;
\draw [shift={(330,111)}, rotate = 47.73] [color={rgb, 255:red, 0; green, 0; blue, 0 }  ][fill={rgb, 255:red, 0; green, 0; blue, 0 }  ][line width=0.75]      (0, 0) circle [x radius= 5.36, y radius= 5.36]   ;
\draw [shift={(209,111)}, rotate = 323.13] [color={rgb, 255:red, 0; green, 0; blue, 0 }  ][line width=0.75]      (0, 0) circle [x radius= 5.36, y radius= 5.36]   ;
\draw    (274.22,114.52) .. controls (310.73,151.6) and (402.8,160.47) .. (461,111) ;
\draw [shift={(461,111)}, rotate = 319.64] [color={rgb, 255:red, 0; green, 0; blue, 0 }  ][fill={rgb, 255:red, 0; green, 0; blue, 0 }  ][line width=0.75]      (0, 0) circle [x radius= 5.36, y radius= 5.36]   ;
\draw [shift={(271,111)}, rotate = 49.64] [color={rgb, 255:red, 0; green, 0; blue, 0 }  ][line width=0.75]      (0, 0) circle [x radius= 5.36, y radius= 5.36]   ;
\draw    (330,111) -- (384.64,111) ;
\draw [shift={(389,111)}, rotate = 0] [color={rgb, 255:red, 0; green, 0; blue, 0 }  ][line width=0.75]      (0, 0) circle [x radius= 5.36, y radius= 5.36]   ;
\draw [shift={(330,111)}, rotate = 0] [color={rgb, 255:red, 0; green, 0; blue, 0 }  ][fill={rgb, 255:red, 0; green, 0; blue, 0 }  ][line width=0.75]      (0, 0) circle [x radius= 5.36, y radius= 5.36]   ;
\draw    (393.36,111) -- (461,111) ;
\draw [shift={(461,111)}, rotate = 0] [color={rgb, 255:red, 0; green, 0; blue, 0 }  ][fill={rgb, 255:red, 0; green, 0; blue, 0 }  ][line width=0.75]      (0, 0) circle [x radius= 5.36, y radius= 5.36]   ;
\draw [shift={(389,111)}, rotate = 0] [color={rgb, 255:red, 0; green, 0; blue, 0 }  ][line width=0.75]      (0, 0) circle [x radius= 5.36, y radius= 5.36]   ;
\draw    (461,111) -- (535.64,111) ;
\draw [shift={(540,111)}, rotate = 0] [color={rgb, 255:red, 0; green, 0; blue, 0 }  ][line width=0.75]      (0, 0) circle [x radius= 5.36, y radius= 5.36]   ;
\draw [shift={(461,111)}, rotate = 0] [color={rgb, 255:red, 0; green, 0; blue, 0 }  ][fill={rgb, 255:red, 0; green, 0; blue, 0 }  ][line width=0.75]      (0, 0) circle [x radius= 5.36, y radius= 5.36]   ;
\draw  [dash pattern={on 4.5pt off 4.5pt}]  (48,62) -- (49,173) ;
\draw  [dash pattern={on 4.5pt off 4.5pt}]  (498,70) -- (499,166) ;
\draw  [fill={rgb, 255:red, 155; green, 155; blue, 155 }  ,fill opacity=1 ] (203,110.5) .. controls (203,106.91) and (205.91,104) .. (209.5,104) .. controls (213.09,104) and (216,106.91) .. (216,110.5) .. controls (216,114.09) and (213.09,117) .. (209.5,117) .. controls (205.91,117) and (203,114.09) .. (203,110.5) -- cycle ;
\draw  [fill={rgb, 255:red, 155; green, 155; blue, 155 }  ,fill opacity=1 ] (264.5,111) .. controls (264.5,107.41) and (267.41,104.5) .. (271,104.5) .. controls (274.59,104.5) and (277.5,107.41) .. (277.5,111) .. controls (277.5,114.59) and (274.59,117.5) .. (271,117.5) .. controls (267.41,117.5) and (264.5,114.59) .. (264.5,111) -- cycle ;
\draw  [fill={rgb, 255:red, 155; green, 155; blue, 155 }  ,fill opacity=1 ] (383,110.5) .. controls (383,106.91) and (385.91,104) .. (389.5,104) .. controls (393.09,104) and (396,106.91) .. (396,110.5) .. controls (396,114.09) and (393.09,117) .. (389.5,117) .. controls (385.91,117) and (383,114.09) .. (383,110.5) -- cycle ;

\draw (60,67.4) node [anchor=north west][inner sep=0.75pt]    {$x_{1}$};
\draw (61,131.4) node [anchor=north west][inner sep=0.75pt]    {$x_{2}$};
\draw (137,114.4) node [anchor=north west][inner sep=0.75pt]    {$-y_{1}$};
\draw (473,86.4) node [anchor=north west][inner sep=0.75pt]    {$-x_{3}$};
\draw (333,154.4) node [anchor=north west][inner sep=0.75pt]  [font=\normalsize]  {$-x_3-w-c_{1}^{r} +c_{1}^{l}$};
\draw (151,158.4) node [anchor=north west][inner sep=0.75pt]    {$-x_3-w-c_{1}^{r} +c_{1}^{l}$};
\draw (420,86.4) node [anchor=north west][inner sep=0.75pt]    {$w$};
\draw (350,86.4) node [anchor=north west][inner sep=0.75pt]    {$w$};
\draw (235,63.4) node [anchor=north west][inner sep=0.75pt]  [font=\small]  {$c_{2}^{l} -c_{2}^{r} +w$};
\draw (120,69.4) node [anchor=north west][inner sep=0.75pt]  [font=\small]  {$c_{2}^{l} -c_{2}^{r} +w$};

\end{tikzpicture}
\caption{}
\label{fig-floorex3}
\end{figure}
\end{center} 
\end{example}

\begin{example}\label{example-compute_invariant}
    Let us consider $\textbf c^r=2$, $\textbf c^l=(-1,0)$, $g=0$, $\alpha=00001$, $\beta=1$, $\tilde\alpha=1$ and $\tilde\beta=0$. In \cref{fig-floorex4}, we see all the floor diagrams that contribute to $N_{(2;-1,0),0}^{00001,1,1,0}(1;2;1,1)$, so by \cref{thm4} we get
    \begin{equation}
        N_{(2;-1,0),0}^{00001,1,1,0}(1;2;1,1)=3\mu(\mathcal{D}_1)+\mu(\mathcal{D}_4)+3\mu(\mathcal{D}_5)+\mu(\mathcal{D}_8)=27+16+12+9=64.
    \end{equation}
    {Note that the floor diagrams $\mathcal D_2$ and $\mathcal D_3$ are obtained from $\mathcal D_1$ by moving the white vertex to the right. This operation create new floor diagrams having the same multiplicity and it explains the factor $3$ in front of $\mu(\mathcal D_1)$. Same reasoning holds for $\mathcal D_5,\mathcal D_6$ and $\mathcal D_7$.}
\end{example}

\section{Proof of \cref{thm1}}\label{pp}

In this section, {we provide a proof of \cref{thm1} and illustrate our results in an example in \cref{sec:example}. We will refer to some results on weighted Ehrhart theory that are listed (and in some cases proved) in \cite{ardila2017double}.}




\begin{figure}[H]
    \tikzset{every picture/.style={line width=0.75pt}} 
\begin{center}    


\end{center}

\caption{}
\label{fig-floorex4}
\end{figure}

\begin{proof}[\textbf{Proof of \cref{thm1}}]\label{thm-proof1}    By \cref{thm4} we have
    \begin{equation}
        F_{(\textbf{d}^r,\textbf{d}^l),\textbf c,g}^{n_1,n_2}(\textbf x,\textbf y)=\sum_{\mathcal D}\mu(\mathcal D)
    \end{equation}
    where the sum runs over all floor diagrams $\mathcal D$ for $S(\textbf{c})$ having multidegree \textbf{d}, genus $g$, divergence multiplicity vector $(\alpha,\beta,\tilde\alpha,\tilde\beta)$ and left-right sequence \textbf{x}. {We proceed by three steps: the first two are essentially the same as in \cite[Theorem 1.3]{ardila2017double}, however in the last step we have to be a bit more careful in saying which are the correct chambers in which the function is polynomial.}\\ 
    {\textbf{Step $1$: rewrite the function $F_{(\textbf{d}^r,\textbf{d}^l),\textbf c,g}^{n_1,n_2}(\textbf x,\textbf y)$ in a more convenient way.}} Let $\tilde{\mathcal D}$ be the graph obtained by removing all weights of $\mathcal D$, but such that the underlying graph $\tilde{\mathcal D}$ inherit the partition $V=L\cup C\cup R$ of the vertices, the ordering of $C$ and the coloring of the vertices. The collection $\mathcal G$ of such graphs that contribute to $F_{(\textbf{d}^r,\textbf{d}^l),\textbf c,g}^{n_1,n_2}(\textbf x,\textbf y)$ is finite and depends only on $g$, $a=\displaystyle\sum_{i=1}^nd_i^r=\sum_{j=1}^md_j^l$ and $n_1+n_2$. Let us denote by $Perm(D_r)$ and $Perm(D_l)$ the sets of permutations of the multisets $D_r$ and $D_l$ respectively and let $r\in Perm(D_r)$ and $l\in Perm(D_l)$. For each graph $G\in\mathcal G$, let $E(G)$ and $V(G)$ be the sets of edges and vertices of $G$ respectively and define the set $W_{G,\textbf{c},r-l}(\textbf{x},\textbf{y})$ of weights $w:E(G)\to\mathbb N$ for which the resulting weighted graph is a floor diagram for $S(\textbf{c})$, i.e. such that the $i$-th black vertex has divergence $r_i-l_i$ and every grey vertex has divergence $0$, with white divergence sequence $(\textbf{x},\textbf{y})$. Note that, by construction, the obtained floor diagram has genus $g$ and multidegree \textbf{d}. Call $\mathbb R^{X}=\{\textbf w:X\to\mathbb R\}$ and let $\pi_{\text{int}}:\mathbb R^{E(G)}\to\mathbb R$ be the polynomial function defined by $\pi_{\text{int}}(\textbf{w})=\displaystyle\prod_{\substack{e\in E(G)\\e \text{ internal}}}\textbf w(e)$, which is the multiplicity of the floor diagram $\mathcal D$ obtained from $G$ adding the weights $\textbf{w}(e)$ at every internal edge $e$. Define
    \begin{equation}
        F_{G,\textbf{c},r-l}(\textbf{x},\textbf{y})=\sum_{\textbf w\in W_{G,\textbf{c},r-l}}\pi_{\text{int}}(\textbf w).
    \end{equation}
    Note that $F_{G,\textbf{c},r-l}(\textbf{x},\textbf{y})$ depends on the order of the entries of \textbf{y}, while in $F_{(\textbf{d}^r,\textbf{d}^l),\textbf c,g}^{n_1,n_2}(\textbf x,\textbf y)$ we have to consider all the distinct orders for $\textbf{y}$:
    \begin{equation}
        F_{(\textbf{d}^r,\textbf{d}^l),\textbf c,g}^{n_1,n_2}(\textbf x,\textbf y)=\frac{1}{\beta_1!\beta_2!\cdots\tilde\beta_1!\tilde\beta_2!\cdots}\sum_{G\in\mathcal G}\sum_{(r,l)}\sum_{\sigma\in S_{n_2}}F_{G,\textbf{c},r-l}(\textbf{x},\sigma(\textbf{y}))
    \end{equation}
    {\textbf{Step $2$: express $F_{(\textbf{d}^r,\textbf{d}^l),\textbf c,g}^{n_1,n_2}(\textbf x,\textbf y)$ as a weighted partition function.}} Recall that the divergence of a vertex is defined as $$\text{div}(v)=\sum_{e:v\to v'}w(e)-\sum_{e:v'\to v}w(e)$$ and that the adjacency matrix of the graph $G$ is given by $A\in\mathbb R^{V(G)\times E(G)}$ which is, in our convention:
    \begin{equation}
        A(v,e)=
        \begin{cases}
            1 &\text{when}\>e:v\to v'\>\text{for some}\>v'\\
            -1 &\text{when}\>e:v'\to v\>\text{for some}\>v'\\
            0 &\text{otherwise}
        \end{cases}
    \end{equation}
    Note that the columns of the matrix $A$ are a subset of the root system $A_{|E(G)|-1}$ (see \cite[Example 4.4]{ardila2017double}), therefore the matrix $A$ is unimodular. Now, take $\textbf{k}\in\mathbb R^{V(G)}$ and define the flow polytope
    \begin{align}
        \Phi_G(\textbf{k})&=\{\textbf w\in\mathbb R^{E(G)}|\>\textbf w(e)\geq0\>\text{for all}\> e\in E(G),\>\text{div}(v)=\textbf k(v)\>\text{for all vertices}\>v\}\\
        &=\{\textbf w\in\mathbb R^{E(G)}|\>A\textbf w=\textbf{k},\>\textbf w\geq0\}.
    \end{align}
    If we take $\textbf{k}$ to be the vector which entries are given by $(\textbf{x},\textbf{y})$ for the white vertices, $r-l$ for the black vertices and $0$ for the grey vertices, then $W_{G,\textbf{c},r-l}(\textbf{x},\textbf{y})=\Phi_G(\textbf{k}){\cap\mathbb Z^{E(G)}}$.\\
    {Define the weighted partition function} $$\mathcal P_{G,\pi_{\text{int}}}(\textbf{k})=\sum_{\textbf w\in\Phi_G(\textbf{k}){\cap\mathbb Z^{E(G)}}}\pi_{\text{int}}(\textbf w)$$ {and consider the hyperplane $\{\textbf{k}\in\mathbb R^{V(G)}|\sum\textbf k(v)=0\}$. Let $H_{r-l}\subset\{\textbf{k}\in\mathbb R^{V(G)}|\sum\textbf k(v)=0\}$ be the subspace determined by the equations}
    \begin{equation}
        \textbf k(w_i)=x_i,\quad \textbf k(w_j)=y_j,\quad \textbf k(u)=0\>\text{for all grey}\>u,\quad \textbf k(b_i)=r_i-l_i\>\text{for all black}\> b_i.
    \end{equation}
    {We have that the restriction of $\mathcal P_{G,\pi_{\text{int}}}(\textbf{k})$ to the subspace $H_{r-l}$ is the function $F_{G,\textbf{c},r-l}(\textbf{x},\textbf{y})$.}\\
    {\textbf{Step $3$: the function $F_{(\textbf{d}^r,\textbf{d}^l),\textbf c,g}^{n_1,n_2}(\textbf x,\textbf y)$ is piecewise polynomial.} By \cite[Theorem 4.2, Example 4.4]{ardila2017double} the weighted partition function $\mathcal P_{G,\pi_{\text{int}}}(\textbf{k})$ is piecewise polynomial relative to the chambers of the discriminant arrangement in $\{\textbf{k}\in\mathbb R^{V(G)}|\sum\textbf k(v)=0\}$. Recall that this arrangement consists of the hyperplanes $\displaystyle\sum_{v'\in V'}\textbf k(v')=0$ for all subsets $V'\subseteq V$. In particular, $F_{G,\textbf{c},r-l}(\textbf{x},\textbf{y})$ is piecewise polynomial relative to the chambers of the discriminant arrangement in $H_{r-l}$. We denote this discriminant arrangement by $S_{r-l}$.} When we symmetrise, the result $\displaystyle\sum_{\sigma\in S_{n_2}}F_{G,\textbf{c},r-l}(\textbf{x},\sigma(\textbf{y}))$ is still piecewise polynomial relative to the same chambers, since the chamber structure is fixed under permutation of the $n_2$ \textbf{y} variables. What remains to prove is that $\displaystyle\sum_{(r,l)}\sum_{\sigma\in S_{n_2}}F_{G,\textbf{c},r-l}(\textbf{x},\sigma(\textbf{y}))$ is piecewise polynomial. {In order to do so, we note that, in general, $S_{r-l}$ is not the same hyperplanes arrangement as $S_{\tilde r-\tilde l}$ for $r\neq\tilde r$ and $l\neq\tilde l$ (see \cref{example-distinct_arrangements}). Therefore, when we sum over all the pairs $(r,l)$, the resulting function will be piecewise polynomial relative to the chambers of the common refinement of the hyperplanes arrangements $S_{r-l}$, in other words $\mathcal{H}^{n_1,n_2}(\textbf{c})=\displaystyle\bigcup_{(r,l)}S_{r-l}$.} This completes the proof. Finally, {the second part of the statement} follows from the same arguments.
\end{proof}
\begin{example}\label{example-distinct_arrangements}
    {We provide an example in which there are permutations $r_1,r_2\in Perm(D_r)$ and $l_1,l_2\in Perm(D_l)$ such that $S_{r_1-l_1}$ and $S_{r_2-l_2}$ do not have the same hyperplanes: let us consider $n_1=n_2=1$, $\textbf{d}^r=(1,2)$, $\textbf{d}^l=(1,1,1)$, $\textbf{c}^r=(c_1^r,c_2^r)$ and $\textbf{c}^l=(c_1^l,c_2^l,c_3^l)$ such that $c_1^r>c_2^r$ and $c_1^l<c_2^l<c_3^l$. The multisets $D_r$ and $D_l$ have the following form:}

    \begin{equation}
        {D_r=\{c_1^r,c_2^r,c_2^r\}\qquad\qquad D_l=\{c_1^l,c_2^l,c_3^l\}.}
    \end{equation}
    {Take $r_1=(c_1^r,c_2^r,c_2^r)$, $r_2=(c_2^r,c_2^r,c_1^r)$, $l_1=(c_1^l,c_2^l,c_3^l)$ and $l_2=(c_1^l,c_3^l,c_2^l)$, so that}

    \begin{equation}
        {r_1-l_1=(c_1^r-c_1^l,c_2^r-c_2^l,c_2^r-c_3^l)\qquad\qquad r_2-l_2=(c_2^r-c_1^l,c_2^r-c_3^l,c_1^r-c_2^l).}
    \end{equation}
    {It is easy to check that the hyperplane $x_1+y_1+2c_2^r-c_2^l-c_3^l=0$ belongs to the hyperplane arrangement $S_{r_1-l_1}$, but not in $S_{r_2-l_2}$.}
\end{example}

\begin{remark}\label{rmk4}
    In \cref{example4} we computed the polytope $\Phi_G(\textbf{k})=[\max\{0,c_2^r-c_2^l\},c_1^l-c_1^r-x_3]$, where $\textbf{k}=(x_1,x_2,c_1^r-c_3^l,y_1,0,0,c_2^r-c_2^l,0,c_1^r-c_1^l,x_3)$. 
\end{remark}

\begin{remark}\label{rmk5}
    In the case in which $m=1$ something interesting happens to the hyperplane arrangements $S_{r-l}$. Indeed, in this case the only left permutation possible is the identity; therefore $S_{r-id}=S_{\tilde r-id}$ with $r,\tilde r\in Perm(D_r)$. This means that the chamber structure for the piecewise polynomiality of the function $F_{(\textbf d^r,\textbf d^l),\textbf c,g}^{n_1,n_2}(\textbf x,\textbf y)$ does not depend on the permutation $r$.
\end{remark}

\begin{remark}\label{rmk-wall crossing}
    In \cite{cavalieri2010tropical} the authors presented a technique to get a wall crossing formula for genus $0$ double Hurwitz numbers and further it was generalised to arbitrary genus in \cite{cavalieri2011wall}. We briefly sketch the technique in the case of genus $0$ double Hurwitz numbers: let $C_1$ and $C_2$ be two adjacent chambers of polynomiality for genus $0$ double Hurwitz numbers and call $\delta=0$ the equation of the wall dividing them. Let us consider a graph contributing to $C_1$ that presents an edge with weight $\delta$. Once we pass through the wall what happens is that the orientation of the edge having weight $\delta$ will be inverted, which corresponds to cutting and regluing the edge in a suitable way to get the new graph and changing the sign of the edge weight. This operation divides the graph into two new graphs contributing to two genus $0$ double Hurwitz numbers having new data and this provides the recursive formula. When adapting this to our case, each edge of weight $\delta$ arises twice, thus we obtain a contribution of $\delta^2$. The squaring  erases the sign change. Therefore, we expect that new techniques may be necessary to possibly derive wall-crossing formulae for our invariants.
\end{remark}

\section{Example}\label{sec:example}
The goal of this section is to explicitly compute the function $F^{n_1,n_2}_{(\textbf{d}^r,\textbf{d}^l),\textbf{c},g}(\textbf{x},\textbf{y})$ in the case $n_1=2$, $n_2=1$, $\textbf{c}^r=(k,0)$, where $k>0$ is an integer, $\textbf{c}^l=0$, $\textbf{d}^r=(1,1)$, $\textbf{d}^l=2$ and $g\geq0$. The corresponding polytope is given in \cref{fig-longexpoly}.

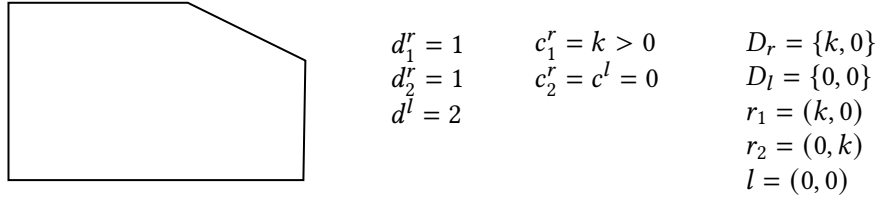
\begin{figure}

    \centering
    
    \tikzset{every picture/.style={line width=0.75pt}} 

\begin{tikzpicture}[x=0.75pt,y=0.75pt,yscale=-1,xscale=1]

\draw   (211,120) -- (210,180) -- (62,180) -- (62,91) -- (152,91) -- cycle ;

\draw (246,102.4) node [anchor=north west][inner sep=0.75pt]    {$ \begin{array}{l}
d_{1}^{r} =1\\
d_{2}^{r} =1\\
d^{l} =2
\end{array}$};
\draw (318,101.4) node [anchor=north west][inner sep=0.75pt]    {$ \begin{array}{l}
c_{1}^{r} =k >0\\
c_{2}^{r} =c^{l} =0
\end{array}$};
\draw (424,101.4) node [anchor=north west][inner sep=0.75pt]    {$ \begin{array}{l}
D_{r} =\{k,0\}\\
D_{l} =\{0,0\}\\
r_{1} =( k,0)\\
r_{2} =( 0,k)\\
l=( 0,0)
\end{array}$};

\end{tikzpicture}

    \caption{The $h$-transverse polygon associated to the data above.}
    \label{fig-longexpoly}
\end{figure}



In \cref{fig:table1,fig:table2} we list all the floor diagrams contributing to $F^{2,1}_{(1,1;2),(k,0;0),0}(x_1,x_2,y_1)$ with divergence sequence $r_1-l$ and $r_2-l$ respectively. Note that, since $c_2^r=c^l=0$ the domain will be

\begin{equation}
    \Lambda=\{(\textbf x,\textbf y)\in\mathbb Z^2\times\mathbb Z|x_1+x_2+y_1+k=0\}.
\end{equation}

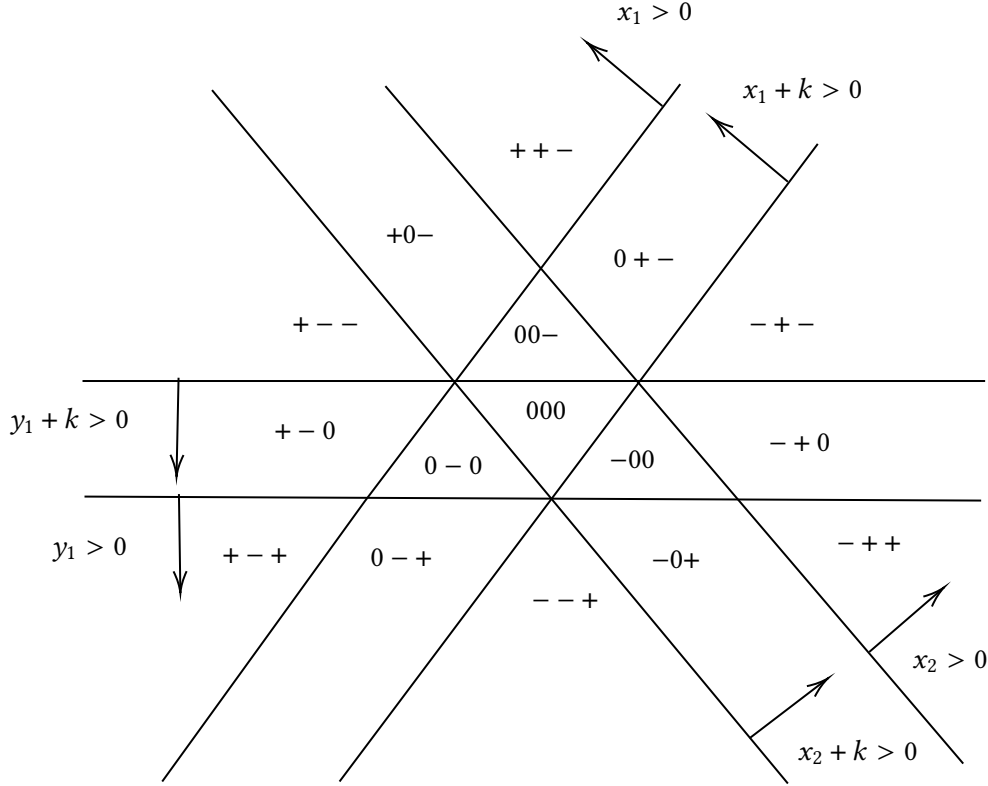
\begin{figure}
    \centering
    \tikzset{every picture/.style={line width=0.75pt}} 

\begin{tikzpicture}[x=0.75pt,y=0.75pt,yscale=-1,xscale=1]

\draw    (390,452) -- (491.71,311.3) -- (650,102) ;
\draw    (479,452) -- (719,132) ;
\draw    (350,251) -- (803,251) ;
\draw    (351,309) -- (801,311) ;
\draw    (792,443) -- (502,103) ;
\draw    (704,453) -- (415,105) ;
\draw    (641,113) -- (604.53,82.29) ;
\draw [shift={(603,81)}, rotate = 40.1] [color={rgb, 255:red, 0; green, 0; blue, 0 }  ][line width=0.75]    (10.93,-3.29) .. controls (6.95,-1.4) and (3.31,-0.3) .. (0,0) .. controls (3.31,0.3) and (6.95,1.4) .. (10.93,3.29)   ;
\draw    (704,151) -- (667.53,120.29) ;
\draw [shift={(666,119)}, rotate = 40.1] [color={rgb, 255:red, 0; green, 0; blue, 0 }  ][line width=0.75]    (10.93,-3.29) .. controls (6.95,-1.4) and (3.31,-0.3) .. (0,0) .. controls (3.31,0.3) and (6.95,1.4) .. (10.93,3.29)   ;
\draw    (744.83,387) -- (780.91,355.83) ;
\draw [shift={(782.42,354.53)}, rotate = 139.18] [color={rgb, 255:red, 0; green, 0; blue, 0 }  ][line width=0.75]    (10.93,-3.29) .. controls (6.95,-1.4) and (3.31,-0.3) .. (0,0) .. controls (3.31,0.3) and (6.95,1.4) .. (10.93,3.29)   ;
\draw    (684.76,430.28) -- (722.64,401.33) ;
\draw [shift={(724.23,400.12)}, rotate = 142.62] [color={rgb, 255:red, 0; green, 0; blue, 0 }  ][line width=0.75]    (10.93,-3.29) .. controls (6.95,-1.4) and (3.31,-0.3) .. (0,0) .. controls (3.31,0.3) and (6.95,1.4) .. (10.93,3.29)   ;
\draw    (398.2,249.35) -- (397.22,297.02) ;
\draw [shift={(397.18,299.02)}, rotate = 271.18] [color={rgb, 255:red, 0; green, 0; blue, 0 }  ][line width=0.75]    (10.93,-3.29) .. controls (6.95,-1.4) and (3.31,-0.3) .. (0,0) .. controls (3.31,0.3) and (6.95,1.4) .. (10.93,3.29)   ;
\draw    (398.39,307.74) -- (399.09,355.42) ;
\draw [shift={(399.12,357.42)}, rotate = 269.16] [color={rgb, 255:red, 0; green, 0; blue, 0 }  ][line width=0.75]    (10.93,-3.29) .. controls (6.95,-1.4) and (3.31,-0.3) .. (0,0) .. controls (3.31,0.3) and (6.95,1.4) .. (10.93,3.29)   ;

\draw (563,130.4) node [anchor=north west][inner sep=0.75pt]    {$++-$};
\draw (615,183.4) node [anchor=north west][inner sep=0.75pt]    {$0+-$};
\draw (683,217.4) node [anchor=north west][inner sep=0.75pt]    {$-+-$};
\draw (693,275.4) node [anchor=north west][inner sep=0.75pt]    {$-+0$};
\draw (571,259.4) node [anchor=north west][inner sep=0.75pt]    {$000$};
\draw (565,221.4) node [anchor=north west][inner sep=0.75pt]    {$00-$};
\draw (520,287.4) node [anchor=north west][inner sep=0.75pt]    {$0-0$};
\draw (613,284.4) node [anchor=north west][inner sep=0.75pt]    {$-00$};
\draw (501,171.4) node [anchor=north west][inner sep=0.75pt]    {$+0-$};
\draw (454,217.4) node [anchor=north west][inner sep=0.75pt]    {$+--$};
\draw (445,269.4) node [anchor=north west][inner sep=0.75pt]    {$+-0$};
\draw (419,333.4) node [anchor=north west][inner sep=0.75pt]    {$+-+$};
\draw (493,333.4) node [anchor=north west][inner sep=0.75pt]    {$0-+$};
\draw (574,356.4) node [anchor=north west][inner sep=0.75pt]    {$--+$};
\draw (634,334.4) node [anchor=north west][inner sep=0.75pt]    {$-0+$};
\draw (727,327.4) node [anchor=north west][inner sep=0.75pt]    {$-++$};
\draw (617,59.4) node [anchor=north west][inner sep=0.75pt]    {$x_{1}  >0$};
\draw (680,97.4) node [anchor=north west][inner sep=0.75pt]    {$x_{1} +k >0$};
\draw (765.26,384.5) node [anchor=north west][inner sep=0.75pt]    {$x_{2}  >0$};
\draw (707.69,429.26) node [anchor=north west][inner sep=0.75pt]    {$x_{2} +k >0$};
\draw (311.19,263.02) node [anchor=north west][inner sep=0.75pt]    {$y_{1} +k >0$};
\draw (332.4,327.63) node [anchor=north west][inner sep=0.75pt]    {$y_{1}  >0$};

\end{tikzpicture}

    \caption{The chamber complex for $F^{2,1}_{(1,1;2),(k,0;0),0}(x_1,x_2,y_1)$}
    \label{fig:chambers}
\end{figure}

The following planes provide the subdivision of $\Lambda$ in sixteen chambers, see \cref{fig:chambers}:

\begin{equation}
    x_1=0,\> x_1+k=0,\> x_2=0,\> x_2+k=0,\> y_1=0,\> y_1+k=0.
\end{equation}

\begin{remark}
We note that, even though we do not work with a Hirzebruch surface, the chamber complex we obtain coincides with the one in \cite[Section 6]{ardila2017double}. This is due to the fact that the polygon in \cref{fig-longexpoly} only differs from the polygon of a Hirzebruch surface by a vertical edge.

\end{remark}

As in \cite{ardila2017double}, we label each chamber with a triple $s_{x_1}s_{x_2}s_{y_1}$ where each $s_i$ is $+,0$ or $-$ according to weather the corresponding variable is $>0$, between $-k$ and $0$, or $<-k$ respectively. For instance, the chamber $0+-$ is given by the inequalities:

\begin{equation}
    x_1+k>0>x_1,\> x_2+k>x_2>0,\> y_1<y_1+k<0.
\end{equation}

Since $F^{2,1}_{(1,1;2),(k,0;0),g}(x_1,x_2,y_1)=F^{2,1}_{(1,1;2),(k,0;0),g}(x_2,x_1,y_1)$, it is sufficient to compute this function for $x_1\geq x_2$, therefore we consider just ten of the sixteen chambers and the corresponding polynomials are listed in \cref{fig:polynomial}. The polynomials in the remaining six chambers can be obtained by symmetry.

    
{\footnotesize \begin{table}
        \centering
        
 \scalebox{0.7}{
\begin{tabular}{|p{0.04\textwidth}|p{0.25\textwidth}|p{0.24\textwidth}|p{0.23\textwidth}|p{0.24\textwidth}|}
\hline 
  & \begin{center}
A
\end{center}
 & \begin{center}
B
\end{center}
 & \begin{center}
C
\end{center}
 & \begin{center}
D
\end{center}
 \\
\hline
 \end{tabular}}\vspace{\baselineskip}

      \caption{The floor diagrams with divergence sequence $r_1-l$ that contribute to $F^{2,1}_{(1,1;2),(k,0;0),0}(x_1,x_2,y_1)$}
    \label{fig:table1}
        
        \end{table}}


    \begin{table}
        \centering
        
\scalebox{0.7}{\begin{tabular}{|p{0.04\textwidth}|p{0.25\textwidth}|p{0.24\textwidth}|p{0.23\textwidth}|p{0.24\textwidth}|}
\hline 
  & \begin{center}
A'
\end{center}
 & \begin{center}
B'
\end{center}
 & \begin{center}
C'
\end{center}
 & \begin{center}
D'
\end{center}
 \\
\hline
\end{tabular}}
        \vspace{\baselineskip}

            \caption{The floor diagrams with divergence sequence $r_2-l$ that contribute to $F^{2,1}_{(1,1;2),(k,0;0),0}(x_1,x_2,y_1)$}
    \label{fig:table2}
        \end{table}

Let us discuss the case $g=0$. The graphs listed in \cref{fig:table1,fig:table2} contribute to $F^{2,1}_{(1,1;2),(k,0;0),0}(x_1,x_2,y_1)$ and they are obtained by a careful analysis of the weights.

\begin{example}
    Take the graphs $B2$ and $B'2$. Note that the graphs in $B2$ and $B'2$ are the same, but we have different divergence sequences and for this reason we have different weights:
    
    \begin{equation}
        \tikzset{every picture/.style={line width=0.75pt}} 
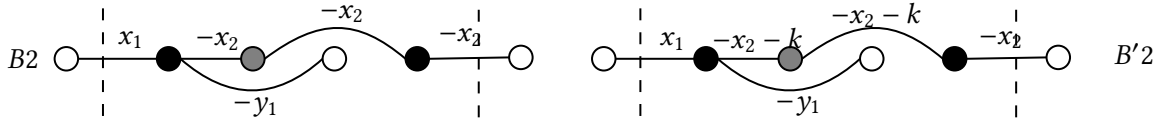
\begin{tikzpicture}[x=0.52pt,y=0.52pt,yscale=-1,xscale=1]

\draw   (38,45.5) .. controls (38,40.81) and (41.81,37) .. (46.5,37) .. controls (51.19,37) and (55,40.81) .. (55,45.5) .. controls (55,50.19) and (51.19,54) .. (46.5,54) .. controls (41.81,54) and (38,50.19) .. (38,45.5) -- cycle ;
\draw  [fill={rgb, 255:red, 0; green, 0; blue, 0 }  ,fill opacity=1 ] (112,45.5) .. controls (112,40.81) and (115.81,37) .. (120.5,37) .. controls (125.19,37) and (129,40.81) .. (129,45.5) .. controls (129,50.19) and (125.19,54) .. (120.5,54) .. controls (115.81,54) and (112,50.19) .. (112,45.5) -- cycle ;
\draw  [fill={rgb, 255:red, 128; green, 128; blue, 128 }  ,fill opacity=1 ] (173,45.5) .. controls (173,40.81) and (176.81,37) .. (181.5,37) .. controls (186.19,37) and (190,40.81) .. (190,45.5) .. controls (190,50.19) and (186.19,54) .. (181.5,54) .. controls (176.81,54) and (173,50.19) .. (173,45.5) -- cycle ;
\draw   (232,45.5) .. controls (232,40.81) and (235.81,37) .. (240.5,37) .. controls (245.19,37) and (249,40.81) .. (249,45.5) .. controls (249,50.19) and (245.19,54) .. (240.5,54) .. controls (235.81,54) and (232,50.19) .. (232,45.5) -- cycle ;
\draw  [fill={rgb, 255:red, 0; green, 0; blue, 0 }  ,fill opacity=1 ] (292,45.5) .. controls (292,40.81) and (295.81,37) .. (300.5,37) .. controls (305.19,37) and (309,40.81) .. (309,45.5) .. controls (309,50.19) and (305.19,54) .. (300.5,54) .. controls (295.81,54) and (292,50.19) .. (292,45.5) -- cycle ;
\draw   (367,44.5) .. controls (367,39.81) and (370.81,36) .. (375.5,36) .. controls (380.19,36) and (384,39.81) .. (384,44.5) .. controls (384,49.19) and (380.19,53) .. (375.5,53) .. controls (370.81,53) and (367,49.19) .. (367,44.5) -- cycle ;
\draw  [dash pattern={on 4.5pt off 4.5pt}]  (73,9) -- (73,90) ;
\draw  [dash pattern={on 4.5pt off 4.5pt}]  (345,10) -- (345,91) ;
\draw    (55,45.5) -- (112,45.5) ;
\draw    (129,45.5) -- (173,45.5) ;
\draw    (309,45.5) -- (367,44.5) ;
\draw    (190,45.5) .. controls (230,15.5) and (257,16) .. (292,45.5) ;
\draw    (129,45.5) .. controls (164,76) and (198,76) .. (232,45.5) ;
\draw   (427,45.5) .. controls (427,40.81) and (430.81,37) .. (435.5,37) .. controls (440.19,37) and (444,40.81) .. (444,45.5) .. controls (444,50.19) and (440.19,54) .. (435.5,54) .. controls (430.81,54) and (427,50.19) .. (427,45.5) -- cycle ;
\draw  [fill={rgb, 255:red, 0; green, 0; blue, 0 }  ,fill opacity=1 ] (501,45.5) .. controls (501,40.81) and (504.81,37) .. (509.5,37) .. controls (514.19,37) and (518,40.81) .. (518,45.5) .. controls (518,50.19) and (514.19,54) .. (509.5,54) .. controls (504.81,54) and (501,50.19) .. (501,45.5) -- cycle ;
\draw  [fill={rgb, 255:red, 128; green, 128; blue, 128 }  ,fill opacity=1 ] (562,45.5) .. controls (562,40.81) and (565.81,37) .. (570.5,37) .. controls (575.19,37) and (579,40.81) .. (579,45.5) .. controls (579,50.19) and (575.19,54) .. (570.5,54) .. controls (565.81,54) and (562,50.19) .. (562,45.5) -- cycle ;
\draw   (621,45.5) .. controls (621,40.81) and (624.81,37) .. (629.5,37) .. controls (634.19,37) and (638,40.81) .. (638,45.5) .. controls (638,50.19) and (634.19,54) .. (629.5,54) .. controls (624.81,54) and (621,50.19) .. (621,45.5) -- cycle ;
\draw  [fill={rgb, 255:red, 0; green, 0; blue, 0 }  ,fill opacity=1 ] (681,45.5) .. controls (681,40.81) and (684.81,37) .. (689.5,37) .. controls (694.19,37) and (698,40.81) .. (698,45.5) .. controls (698,50.19) and (694.19,54) .. (689.5,54) .. controls (684.81,54) and (681,50.19) .. (681,45.5) -- cycle ;
\draw   (756,44.5) .. controls (756,39.81) and (759.81,36) .. (764.5,36) .. controls (769.19,36) and (773,39.81) .. (773,44.5) .. controls (773,49.19) and (769.19,53) .. (764.5,53) .. controls (759.81,53) and (756,49.19) .. (756,44.5) -- cycle ;
\draw  [dash pattern={on 4.5pt off 4.5pt}]  (462,9) -- (462,90) ;
\draw  [dash pattern={on 4.5pt off 4.5pt}]  (734,10) -- (734,91) ;
\draw    (444,45.5) -- (501,45.5) ;
\draw    (518,45.5) -- (562,45.5) ;
\draw    (698,45.5) -- (756,44.5) ;
\draw    (579,45.5) .. controls (619,15.5) and (646,16) .. (681,45.5) ;
\draw    (518,45.5) .. controls (553,76) and (587,76) .. (621,45.5) ;

\draw (82,23.4) node [anchor=north west][inner sep=0.75pt]    {$x_{1}$};
\draw (474,24.4) node [anchor=north west][inner sep=0.75pt]    {$x_{1}$};
\draw (314,21.4) node [anchor=north west][inner sep=0.75pt]    {$-x_{2}$};
\draw (228,4.4) node [anchor=north west][inner sep=0.75pt]    {$-x_{2}$};
\draw (138,24.4) node [anchor=north west][inner sep=0.75pt]    {$-x_{2}$};
\draw (705,23.4) node [anchor=north west][inner sep=0.75pt]    {$-x_{2}$};
\draw (165,71.4) node [anchor=north west][inner sep=0.75pt]    {$-y_{1}$};
\draw (558,70.4) node [anchor=north west][inner sep=0.75pt]    {$-y_{1}$};
\draw (512,20.4) node [anchor=north west][inner sep=0.75pt]    {$-x_{2} -k$};
\draw (598,2.4) node [anchor=north west][inner sep=0.75pt]    {$-x_{2} -k$};
\draw (2,36.4) node [anchor=north west][inner sep=0.75pt]    {$B2$};
\draw (803,32.4) node [anchor=north west][inner sep=0.75pt]    {$B'2$};

\end{tikzpicture}
    \end{equation}
    Therefore, the graph $B2$ contributes to $F^{2,1}_{(1,1;2),(k,0;0),0}(x_1,x_2,y_1)$ with weight $(-y_1)x_2^2$ as long as $x_1>0,x_2,y_1<0$; that is, in chambers $+0-,+--$ and $+-0$. On the other hand, the graph $B'2$ contributes to $F^{2,1}_{(1,1;2),(k,0;0),0}(x_1,x_2,y_1)$ with weight $(-y_1)(x_2+k)^2$ as long as $x_1>0,x_2<x_2+k<0,y_1<0$; that is, in chambers $+--$ and $+-0$.
\end{example}

In general, all the graphs in \cref{fig:table1,fig:table2} contribute to $F^{2,1}_{(1,1;2),(k,0;0),0}(x_1,x_2,y_1)/|y_1|$ with weight $(w+k)^2$, where $w\in\{0,x_1,x_2,x_1-k,x_2-k,y_1,y_1-k\}$.

\begin{remark}
    For each graph in rows from seven to thirteen in \cref{fig:table1} and for each graph in rows from five to nine in \cref{fig:table2} (i.e. when $x_1$ and $x_2$ have the same sign), there are a priori two different possibilities of labeling the vertices in $L$ or $R$, respectively, with $\tilde q_1$ and $\tilde q_2$ or $q_1$ and $q_2$. The two corresponding floor diagrams are the same for graphs in rows eight, ten, eleven and twelve in \cref{fig:table1} and in rows six, seven and eight in \cref{fig:table2}, and they are different for graphs in rows seven, nine and thirteen in \cref{fig:table1} and in rows five and nine in \cref{fig:table2}.
\end{remark}

    \begin{table}
        \centering
        
\begin{tabular}{|p{0.14\textwidth}|p{0.29\textwidth}|p{0.56\textwidth}|}
\hline 
 \begin{center}
\textbf{Chamber}
\end{center}
 & \begin{center}
\textbf{Graphs ($g=0$)}
\end{center}
 & \begin{center}
$\displaystyle F_{( 1,1;2) ,( k,0;0) ,g}^{2,1}( x_{1} ,x_{2} ,y_{1}) /|y_{1} |$
\end{center}
 \\
\hline 
 \begin{center}
$++-$
\end{center}
 & \begin{center}
A'5,A'5,A7,A7,A8
\end{center}
 & \begin{center}
$\begin{array}{cc}
\displaystyle
\Gamma ( y_{1} -k) +\Gamma ( y_{1}) +\Gamma ( x_{1}) +\\
 +\Gamma ( x_{1} -k) +\Gamma ( x_{2}) +\Gamma ( x_{2} -k)\end{array}$
\end{center}
 \\
\hline 
 \begin{center}
$+0-$
\end{center}
 & {\begin{center}
A1,A'1,A2,B2,C2,\\A3,A'3,A4,A5\end{center}
}
 & \begin{center}
$\displaystyle  \begin{array}{{>{\displaystyle}l}}
\Gamma ( x_{1}) +\Gamma ( x_{1} -k) +( g+3) \Gamma ( x_{2} -k) +\\
+\Gamma ( y_{1} -k) +\Gamma ( y_{1}) +\Gamma ( x_{2}) +\Gamma ( 0)
\end{array}$
\end{center}
 \\
\hline 
 \begin{center}
$+--$
\end{center}
 & {\begin{center}
A1,A'1.A2,A'2,B2,\\B'2,C2,C'2,A3,A'3,\\A5
\end{center}
} & \begin{center}
$\begin{array}{cc}
\Gamma ( x_{1}) +\Gamma ( x_{1} -k) +( g+3) \Gamma ( x_{2} -k) +\\
+( g+3) \Gamma ( x_{2}) +\Gamma ( y_{1} -k) +\Gamma ( y_{1}) +\Gamma ( 0)
\end{array}$
\end{center}
 \\
\hline 
 \begin{center}
$00-$
\end{center}
 &{\begin{center}
A'8,A9,A9,A10,A12,\\A13,A13,B13,B13,C13,\\C13
\end{center}
} & \begin{center}
$\displaystyle  \begin{array}{{>{\displaystyle}l}}
\Gamma ( x_{1}) +\Gamma ( x_{2}) +\Gamma ( 0) +\Gamma ( y_{1} -k) +\Gamma ( y_{1}) +\\
+( g+3) \Gamma ( x_{2} -k) +( g+3) \Gamma ( x_{1} -k)
\end{array}$
\end{center}
 \\
\hline 
 \begin{center}
$+-0$
\end{center}
 & {\begin{center}
A1,A'1,A2,A'2,B2,\\B'2,C2,C'2,A3,A5,\\A6,B6,C6
\end{center}
} & \begin{center}
$\displaystyle  \begin{array}{{>{\displaystyle}l}}
\Gamma ( x_{1}) +\Gamma ( x_{1} -k) +( g+3) \Gamma ( x_{2} -k) +\\
+( g+3) \Gamma ( x_{2}) +\Gamma ( y_{1} -k) +\Gamma ( 0) +\\
+( g+3) \Gamma ( y_{1})
\end{array}$
\end{center}
 \\
\hline 
 \begin{center}
$0-0$
\end{center}
 & {\begin{center}
A9,A'9,B'9,C'9,A10,\\A11,B11,C11,A12,A13,\\A13,B13,B13,C13,C13
\end{center}
} & \begin{center}
$\displaystyle  \begin{array}{{>{\displaystyle}l}}
\Gamma ( x_{1}) +\Gamma ( 0) +( g+3) \Gamma ( y_{1}) +\Gamma ( y_{1} -k) +\\
+( g+3) \Gamma ( x_{2} -k) +( g+3) \Gamma ( x_{1} -k) +\\
+( g+3) \Gamma ( x_{2})
\end{array}$
\end{center}
 \\
\hline 
 \begin{center}
$000$
\end{center}
 & {\begin{center}
A9,A9,A10,A11,B11,\\C11,A12,A13,A13,B13,\\B13,C13,C13
\end{center}
} & \begin{center}
$\displaystyle  \begin{array}{{>{\displaystyle}l}}
\Gamma ( x_{1}) +\Gamma ( x_{2}) +\Gamma ( 0) +( g+3) \Gamma ( y_{1}) +\\
+\Gamma ( y_{1} -k) +( g+3) \Gamma ( 0) +\Gamma ( y_{1}) +\Gamma ( y_{1} -k)
\end{array}$
\end{center}
 \\
\hline 
 \begin{center}
$+-+$
\end{center}
 & {\begin{center}
B1,B'1,C1,C'1,D1,\\D'1,D2,D'2,D'4,B5,\\C5,D5,D6
\end{center}
} & \begin{center}
$\displaystyle  \begin{array}{{>{\displaystyle}l}}
( g+3) \Gamma ( x_{1}) +( g+3) \Gamma ( x_{1} -k) +\Gamma ( x_{2} -k) +\\
+\Gamma ( x_{2}) +( g+3) \Gamma ( 0) +\Gamma ( y_{1}) +\Gamma ( y_{1} -k)
\end{array}$
\end{center}
 \\
\hline 
 \begin{center}
$0-+$
\end{center}
 & {\begin{center}
D'7,B9,C9,D9,D'9,\\B10,C10,D10,D11,D13,\\D13
\end{center}
} & \begin{center}
$\displaystyle  \begin{array}{{>{\displaystyle}l}}
( g+3) \Gamma ( x_{1}) +( g+3) \Gamma ( 0) +\Gamma ( y_{1}) +\\
+\Gamma ( y_{1} -k) +\Gamma ( x_{2}) +\Gamma ( x_{2} -k) +\Gamma ( x_{1} -k)
\end{array}$
\end{center}
 \\
\hline 
 \begin{center}
$--+$
\end{center}
 & {\begin{center}
D'7,D'9,D'9,B10,C10,\\D10,D11,D13,D13
\end{center}
} & \begin{center}
$\displaystyle  \begin{array}{{>{\displaystyle}l}}
( g+3) \Gamma ( 0) +\Gamma ( y_{1}) +\Gamma ( y_{1} -k) +\Gamma ( x_{2}) +\\
+\Gamma ( x_{2} -k) +\Gamma ( x_{1}) +\Gamma ( x_{1} -k)
\end{array}$
\end{center}
 \\
 \hline
\end{tabular}\vspace{\baselineskip}

    \caption{The double Gromov-Witten invariants $F^{2,1}_{(1,1;2),(k,0;0),g}(x_1,x_2,y_1)$}
    \label{fig:polynomial}
        
        \end{table}
        
Let us discuss the case $g>0$. In each graph, we need to replace the grey vertex and its two incident edges by $g+1$ grey vertices and the corresponding $2(g+1)$ edges. The position of an intermediate white vertex can be chosen among the $g+1$ grey vertices; there are $g+3$ choices. This gives rise to the factor $g+3$ in \cref{fig:polynomial}. For example, in chamber $+--$ and genus $g=0$ the graphs $A2,B2,C2$ are isomorphic as unoriented graphs, and they account the three possible positions of the white vertex in $C$ relative to the black and grey vertices.

Suppose $w$ to be the weight of a black-grey edge in a genus $0$ graph. In the genus $g$ graph, that total weight $w$ has to be distributed among $g+1$ weights. Therefore the resulting contribution is

\begin{equation}
    \Gamma_g(w)=\sum_{w_1+\dots+w_{g+1}=w}\prod_{i=1}^{g+1}w_i^2
\end{equation}
where $w_1,\dots,w_{g+1}$ are positive integers. The function $\Gamma$ in \cref{fig:polynomial} is given by

\begin{equation}
    \Gamma(w)=\Gamma_g(|w+k|).
\end{equation}
{Now, we want to show through explicit computations the result in \cref{thm2}. Let us consider the chamber $++-$ and $g=1$ and call $G_1$ the underlying graph of $\mathcal{D}_1,\mathcal{D}_2,\mathcal{D}_3,\mathcal{D}_4$ and $G_2$ the underlying graph of $\mathcal D_5$ (see \cref{fig:chamber++-}). We get}

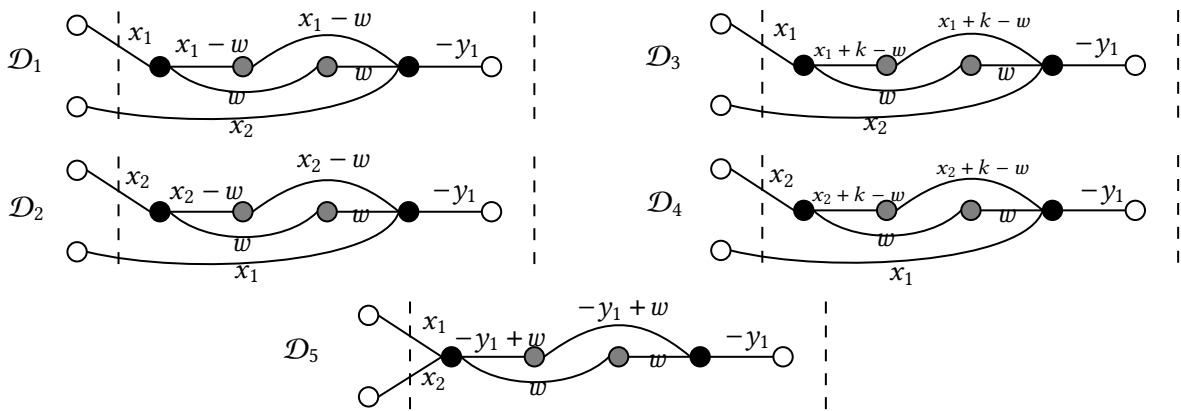
\begin{figure}[H]
    \centering
    \tikzset{every picture/.style={line width=0.75pt}} 

\begin{tikzpicture}[x=0.52pt,y=0.52pt,yscale=-1,xscale=1]

\draw   (83,75) .. controls (83,71.13) and (86.13,68) .. (90,68) .. controls (93.87,68) and (97,71.13) .. (97,75) .. controls (97,78.87) and (93.87,82) .. (90,82) .. controls (86.13,82) and (83,78.87) .. (83,75) -- cycle ;
\draw   (83,16) .. controls (83,12.13) and (86.13,9) .. (90,9) .. controls (93.87,9) and (97,12.13) .. (97,16) .. controls (97,19.87) and (93.87,23) .. (90,23) .. controls (86.13,23) and (83,19.87) .. (83,16) -- cycle ;
\draw  [fill={rgb, 255:red, 0; green, 0; blue, 0 }  ,fill opacity=1 ] (143,46) .. controls (143,42.13) and (146.13,39) .. (150,39) .. controls (153.87,39) and (157,42.13) .. (157,46) .. controls (157,49.87) and (153.87,53) .. (150,53) .. controls (146.13,53) and (143,49.87) .. (143,46) -- cycle ;
\draw  [fill={rgb, 255:red, 128; green, 128; blue, 128 }  ,fill opacity=1 ] (203,46) .. controls (203,42.13) and (206.13,39) .. (210,39) .. controls (213.87,39) and (217,42.13) .. (217,46) .. controls (217,49.87) and (213.87,53) .. (210,53) .. controls (206.13,53) and (203,49.87) .. (203,46) -- cycle ;
\draw  [fill={rgb, 255:red, 128; green, 128; blue, 128 }  ,fill opacity=1 ] (264,46) .. controls (264,42.13) and (267.13,39) .. (271,39) .. controls (274.87,39) and (278,42.13) .. (278,46) .. controls (278,49.87) and (274.87,53) .. (271,53) .. controls (267.13,53) and (264,49.87) .. (264,46) -- cycle ;
\draw  [fill={rgb, 255:red, 0; green, 0; blue, 0 }  ,fill opacity=1 ] (323,46) .. controls (323,42.13) and (326.13,39) .. (330,39) .. controls (333.87,39) and (337,42.13) .. (337,46) .. controls (337,49.87) and (333.87,53) .. (330,53) .. controls (326.13,53) and (323,49.87) .. (323,46) -- cycle ;
\draw   (383,46) .. controls (383,42.13) and (386.13,39) .. (390,39) .. controls (393.87,39) and (397,42.13) .. (397,46) .. controls (397,49.87) and (393.87,53) .. (390,53) .. controls (386.13,53) and (383,49.87) .. (383,46) -- cycle ;
\draw  [dash pattern={on 4.5pt off 4.5pt}]  (120,6) -- (120,87) ;
\draw  [dash pattern={on 4.5pt off 4.5pt}]  (421,6) -- (421,87) ;
\draw    (97,16) -- (143,46) ;
\draw    (337,46) -- (383,46) ;
\draw    (157,46) -- (203,46) ;
\draw    (97,75) .. controls (171,93) and (311,84) .. (323,46) ;
\draw    (217,46) .. controls (257,16) and (286,15) .. (323,46) ;
\draw    (157,46) .. controls (183,71) and (240,70) .. (264,46) ;
\draw    (278,46) -- (323,46) ;
\draw   (83,180) .. controls (83,176.13) and (86.13,173) .. (90,173) .. controls (93.87,173) and (97,176.13) .. (97,180) .. controls (97,183.87) and (93.87,187) .. (90,187) .. controls (86.13,187) and (83,183.87) .. (83,180) -- cycle ;
\draw   (83,121) .. controls (83,117.13) and (86.13,114) .. (90,114) .. controls (93.87,114) and (97,117.13) .. (97,121) .. controls (97,124.87) and (93.87,128) .. (90,128) .. controls (86.13,128) and (83,124.87) .. (83,121) -- cycle ;
\draw  [fill={rgb, 255:red, 0; green, 0; blue, 0 }  ,fill opacity=1 ] (143,151) .. controls (143,147.13) and (146.13,144) .. (150,144) .. controls (153.87,144) and (157,147.13) .. (157,151) .. controls (157,154.87) and (153.87,158) .. (150,158) .. controls (146.13,158) and (143,154.87) .. (143,151) -- cycle ;
\draw  [fill={rgb, 255:red, 128; green, 128; blue, 128 }  ,fill opacity=1 ] (203,151) .. controls (203,147.13) and (206.13,144) .. (210,144) .. controls (213.87,144) and (217,147.13) .. (217,151) .. controls (217,154.87) and (213.87,158) .. (210,158) .. controls (206.13,158) and (203,154.87) .. (203,151) -- cycle ;
\draw  [fill={rgb, 255:red, 128; green, 128; blue, 128 }  ,fill opacity=1 ] (264,151) .. controls (264,147.13) and (267.13,144) .. (271,144) .. controls (274.87,144) and (278,147.13) .. (278,151) .. controls (278,154.87) and (274.87,158) .. (271,158) .. controls (267.13,158) and (264,154.87) .. (264,151) -- cycle ;
\draw  [fill={rgb, 255:red, 0; green, 0; blue, 0 }  ,fill opacity=1 ] (323,151) .. controls (323,147.13) and (326.13,144) .. (330,144) .. controls (333.87,144) and (337,147.13) .. (337,151) .. controls (337,154.87) and (333.87,158) .. (330,158) .. controls (326.13,158) and (323,154.87) .. (323,151) -- cycle ;
\draw   (383,151) .. controls (383,147.13) and (386.13,144) .. (390,144) .. controls (393.87,144) and (397,147.13) .. (397,151) .. controls (397,154.87) and (393.87,158) .. (390,158) .. controls (386.13,158) and (383,154.87) .. (383,151) -- cycle ;
\draw  [dash pattern={on 4.5pt off 4.5pt}]  (120,111) -- (120,192) ;
\draw  [dash pattern={on 4.5pt off 4.5pt}]  (421,111) -- (421,192) ;
\draw    (97,121) -- (143,151) ;
\draw    (337,151) -- (383,151) ;
\draw    (157,151) -- (203,151) ;
\draw    (97,180) .. controls (171,198) and (311,189) .. (323,151) ;
\draw    (217,151) .. controls (257,121) and (286,120) .. (323,151) ;
\draw    (157,151) .. controls (183,176) and (240,175) .. (264,151) ;
\draw    (278,151) -- (323,151) ;
\draw   (549,74) .. controls (549,70.13) and (552.13,67) .. (556,67) .. controls (559.87,67) and (563,70.13) .. (563,74) .. controls (563,77.87) and (559.87,81) .. (556,81) .. controls (552.13,81) and (549,77.87) .. (549,74) -- cycle ;
\draw   (549,15) .. controls (549,11.13) and (552.13,8) .. (556,8) .. controls (559.87,8) and (563,11.13) .. (563,15) .. controls (563,18.87) and (559.87,22) .. (556,22) .. controls (552.13,22) and (549,18.87) .. (549,15) -- cycle ;
\draw  [fill={rgb, 255:red, 0; green, 0; blue, 0 }  ,fill opacity=1 ] (609,45) .. controls (609,41.13) and (612.13,38) .. (616,38) .. controls (619.87,38) and (623,41.13) .. (623,45) .. controls (623,48.87) and (619.87,52) .. (616,52) .. controls (612.13,52) and (609,48.87) .. (609,45) -- cycle ;
\draw  [fill={rgb, 255:red, 128; green, 128; blue, 128 }  ,fill opacity=1 ] (669,45) .. controls (669,41.13) and (672.13,38) .. (676,38) .. controls (679.87,38) and (683,41.13) .. (683,45) .. controls (683,48.87) and (679.87,52) .. (676,52) .. controls (672.13,52) and (669,48.87) .. (669,45) -- cycle ;
\draw  [fill={rgb, 255:red, 128; green, 128; blue, 128 }  ,fill opacity=1 ] (730,45) .. controls (730,41.13) and (733.13,38) .. (737,38) .. controls (740.87,38) and (744,41.13) .. (744,45) .. controls (744,48.87) and (740.87,52) .. (737,52) .. controls (733.13,52) and (730,48.87) .. (730,45) -- cycle ;
\draw  [fill={rgb, 255:red, 0; green, 0; blue, 0 }  ,fill opacity=1 ] (789,45) .. controls (789,41.13) and (792.13,38) .. (796,38) .. controls (799.87,38) and (803,41.13) .. (803,45) .. controls (803,48.87) and (799.87,52) .. (796,52) .. controls (792.13,52) and (789,48.87) .. (789,45) -- cycle ;
\draw   (849,45) .. controls (849,41.13) and (852.13,38) .. (856,38) .. controls (859.87,38) and (863,41.13) .. (863,45) .. controls (863,48.87) and (859.87,52) .. (856,52) .. controls (852.13,52) and (849,48.87) .. (849,45) -- cycle ;
\draw  [dash pattern={on 4.5pt off 4.5pt}]  (586,5) -- (586,86) ;
\draw  [dash pattern={on 4.5pt off 4.5pt}]  (887,5) -- (887,86) ;
\draw    (563,15) -- (609,45) ;
\draw    (803,45) -- (849,45) ;
\draw    (623,45) -- (669,45) ;
\draw    (563,74) .. controls (637,92) and (777,83) .. (789,45) ;
\draw    (683,45) .. controls (723,15) and (752,14) .. (789,45) ;
\draw    (623,45) .. controls (649,70) and (706,69) .. (730,45) ;
\draw    (744,45) -- (789,45) ;
\draw   (549,179) .. controls (549,175.13) and (552.13,172) .. (556,172) .. controls (559.87,172) and (563,175.13) .. (563,179) .. controls (563,182.87) and (559.87,186) .. (556,186) .. controls (552.13,186) and (549,182.87) .. (549,179) -- cycle ;
\draw   (549,120) .. controls (549,116.13) and (552.13,113) .. (556,113) .. controls (559.87,113) and (563,116.13) .. (563,120) .. controls (563,123.87) and (559.87,127) .. (556,127) .. controls (552.13,127) and (549,123.87) .. (549,120) -- cycle ;
\draw  [fill={rgb, 255:red, 0; green, 0; blue, 0 }  ,fill opacity=1 ] (609,150) .. controls (609,146.13) and (612.13,143) .. (616,143) .. controls (619.87,143) and (623,146.13) .. (623,150) .. controls (623,153.87) and (619.87,157) .. (616,157) .. controls (612.13,157) and (609,153.87) .. (609,150) -- cycle ;
\draw  [fill={rgb, 255:red, 128; green, 128; blue, 128 }  ,fill opacity=1 ] (669,150) .. controls (669,146.13) and (672.13,143) .. (676,143) .. controls (679.87,143) and (683,146.13) .. (683,150) .. controls (683,153.87) and (679.87,157) .. (676,157) .. controls (672.13,157) and (669,153.87) .. (669,150) -- cycle ;
\draw  [fill={rgb, 255:red, 128; green, 128; blue, 128 }  ,fill opacity=1 ] (730,150) .. controls (730,146.13) and (733.13,143) .. (737,143) .. controls (740.87,143) and (744,146.13) .. (744,150) .. controls (744,153.87) and (740.87,157) .. (737,157) .. controls (733.13,157) and (730,153.87) .. (730,150) -- cycle ;
\draw  [fill={rgb, 255:red, 0; green, 0; blue, 0 }  ,fill opacity=1 ] (789,150) .. controls (789,146.13) and (792.13,143) .. (796,143) .. controls (799.87,143) and (803,146.13) .. (803,150) .. controls (803,153.87) and (799.87,157) .. (796,157) .. controls (792.13,157) and (789,153.87) .. (789,150) -- cycle ;
\draw   (849,150) .. controls (849,146.13) and (852.13,143) .. (856,143) .. controls (859.87,143) and (863,146.13) .. (863,150) .. controls (863,153.87) and (859.87,157) .. (856,157) .. controls (852.13,157) and (849,153.87) .. (849,150) -- cycle ;
\draw  [dash pattern={on 4.5pt off 4.5pt}]  (586,110) -- (586,191) ;
\draw  [dash pattern={on 4.5pt off 4.5pt}]  (887,110) -- (887,191) ;
\draw    (563,120) -- (609,150) ;
\draw    (803,150) -- (849,150) ;
\draw    (623,150) -- (669,150) ;
\draw    (563,179) .. controls (637,197) and (777,188) .. (789,150) ;
\draw    (683,150) .. controls (723,120) and (752,119) .. (789,150) ;
\draw    (623,150) .. controls (649,175) and (706,174) .. (730,150) ;
\draw    (744,150) -- (789,150) ;
\draw   (294,285) .. controls (294,281.13) and (297.13,278) .. (301,278) .. controls (304.87,278) and (308,281.13) .. (308,285) .. controls (308,288.87) and (304.87,292) .. (301,292) .. controls (297.13,292) and (294,288.87) .. (294,285) -- cycle ;
\draw   (294,226) .. controls (294,222.13) and (297.13,219) .. (301,219) .. controls (304.87,219) and (308,222.13) .. (308,226) .. controls (308,229.87) and (304.87,233) .. (301,233) .. controls (297.13,233) and (294,229.87) .. (294,226) -- cycle ;
\draw  [fill={rgb, 255:red, 0; green, 0; blue, 0 }  ,fill opacity=1 ] (354,256) .. controls (354,252.13) and (357.13,249) .. (361,249) .. controls (364.87,249) and (368,252.13) .. (368,256) .. controls (368,259.87) and (364.87,263) .. (361,263) .. controls (357.13,263) and (354,259.87) .. (354,256) -- cycle ;
\draw  [fill={rgb, 255:red, 128; green, 128; blue, 128 }  ,fill opacity=1 ] (414,256) .. controls (414,252.13) and (417.13,249) .. (421,249) .. controls (424.87,249) and (428,252.13) .. (428,256) .. controls (428,259.87) and (424.87,263) .. (421,263) .. controls (417.13,263) and (414,259.87) .. (414,256) -- cycle ;
\draw  [fill={rgb, 255:red, 128; green, 128; blue, 128 }  ,fill opacity=1 ] (475,256) .. controls (475,252.13) and (478.13,249) .. (482,249) .. controls (485.87,249) and (489,252.13) .. (489,256) .. controls (489,259.87) and (485.87,263) .. (482,263) .. controls (478.13,263) and (475,259.87) .. (475,256) -- cycle ;
\draw  [fill={rgb, 255:red, 0; green, 0; blue, 0 }  ,fill opacity=1 ] (534,256) .. controls (534,252.13) and (537.13,249) .. (541,249) .. controls (544.87,249) and (548,252.13) .. (548,256) .. controls (548,259.87) and (544.87,263) .. (541,263) .. controls (537.13,263) and (534,259.87) .. (534,256) -- cycle ;
\draw   (594,256) .. controls (594,252.13) and (597.13,249) .. (601,249) .. controls (604.87,249) and (608,252.13) .. (608,256) .. controls (608,259.87) and (604.87,263) .. (601,263) .. controls (597.13,263) and (594,259.87) .. (594,256) -- cycle ;
\draw  [dash pattern={on 4.5pt off 4.5pt}]  (331,216) -- (331,297) ;
\draw  [dash pattern={on 4.5pt off 4.5pt}]  (632,216) -- (632,297) ;
\draw    (308,226) -- (354,256) ;
\draw    (548,256) -- (594,256) ;
\draw    (368,256) -- (414,256) ;
\draw    (428,256) .. controls (468,226) and (497,225) .. (534,256) ;
\draw    (368,256) .. controls (394,281) and (451,280) .. (475,256) ;
\draw    (489,256) -- (534,256) ;
\draw    (308,285) -- (354,256) ;

\draw (125,15.4) node [anchor=north west][inner sep=0.75pt]    {$x_{1}$};
\draw (202,187.4) node [anchor=north west][inner sep=0.75pt]    {$x_{1}$};
\draw (592,12.4) node [anchor=north west][inner sep=0.75pt]    {$x_{1}$};
\draw (676,189.4) node [anchor=north west][inner sep=0.75pt]    {$x_{1}$};
\draw (338,224.4) node [anchor=north west][inner sep=0.75pt]    {$x_{1}$};
\draw (198,82.4) node [anchor=north west][inner sep=0.75pt]    {$x_{2}$};
\draw (123,119.4) node [anchor=north west][inner sep=0.75pt]    {$x_{2}$};
\draw (657,80.4) node [anchor=north west][inner sep=0.75pt]    {$x_{2}$};
\draw (589,118.4) node [anchor=north west][inner sep=0.75pt]    {$x_{2}$};
\draw (337,265.4) node [anchor=north west][inner sep=0.75pt]    {$x_{2}$};
\draw (159,24.4) node [anchor=north west][inner sep=0.75pt]    {$x_{1} -w$};
\draw (248,5.4) node [anchor=north west][inner sep=0.75pt]    {$x_{1} -w$};
\draw (198,61.4) node [anchor=north west][inner sep=0.75pt]    {$w$};
\draw (290,44.4) node [anchor=north west][inner sep=0.75pt]    {$w$};
\draw (202,168.4) node [anchor=north west][inner sep=0.75pt]    {$w$};
\draw (288,147.4) node [anchor=north west][inner sep=0.75pt]    {$w$};
\draw (669,61.4) node [anchor=north west][inner sep=0.75pt]    {$w$};
\draw (755,44.4) node [anchor=north west][inner sep=0.75pt]    {$w$};
\draw (668,164.4) node [anchor=north west][inner sep=0.75pt]    {$w$};
\draw (756,147.4) node [anchor=north west][inner sep=0.75pt]    {$w$};
\draw (416,272.4) node [anchor=north west][inner sep=0.75pt]    {$w$};
\draw (503,252.4) node [anchor=north west][inner sep=0.75pt]    {$w$};
\draw (347,24.4) node [anchor=north west][inner sep=0.75pt]    {$-y_{1}$};
\draw (345,129.4) node [anchor=north west][inner sep=0.75pt]    {$-y_{1}$};
\draw (810,23.4) node [anchor=north west][inner sep=0.75pt]    {$-y_{1}$};
\draw (814,129.4) node [anchor=north west][inner sep=0.75pt]    {$-y_{1}$};
\draw (557,235.4) node [anchor=north west][inner sep=0.75pt]    {$-y_{1}$};
\draw (155,130.4) node [anchor=north west][inner sep=0.75pt]    {$x_{2} -w$};
\draw (247,107.4) node [anchor=north west][inner sep=0.75pt]    {$x_{2} -w$};
\draw (621,25.4) node [anchor=north west][inner sep=0.75pt]  [font=\scriptsize]  {$x_{1} +k-w$};
\draw (711,4.4) node [anchor=north west][inner sep=0.75pt]  [font=\scriptsize]  {$x_{1} +k-w$};
\draw (620,130.4) node [anchor=north west][inner sep=0.75pt]  [font=\scriptsize]  {$x_{2} +k-w$};
\draw (710,110.4) node [anchor=north west][inner sep=0.75pt]  [font=\scriptsize]  {$x_{2} +k-w$};
\draw (450,213.4) node [anchor=north west][inner sep=0.75pt]    {$-y_{1} +w$};
\draw (361,235.4) node [anchor=north west][inner sep=0.75pt]    {$-y_{1} +w$};
\draw (37,30.4) node [anchor=north west][inner sep=0.75pt]    {$\mathcal{D}_{1}$};
\draw (38,137.4) node [anchor=north west][inner sep=0.75pt]    {$\mathcal{D}_{2}$};
\draw (499,28.4) node [anchor=north west][inner sep=0.75pt]    {$\mathcal{D}_{3}$};
\draw (500,134.4) node [anchor=north west][inner sep=0.75pt]    {$\mathcal{D}_{4}$};
\draw (237,240.4) node [anchor=north west][inner sep=0.75pt]    {$\mathcal{D}_{5}$};

\end{tikzpicture}

    \caption{The floor diagrams of genus $1$ contributing to $F^{2,1}_{(1,1;2),(k,0;0),g}(x_1,x_2,y_1)$ in the chamber $++-$}
    \label{fig:chamber++-}
\end{figure}

\begin{align}
    F_{G_1,\textbf{c},r_2-l}(x_1,x_2,y_1)&=\sum_{w=0}^{x_1}(-y_1)(x_1-w)^2w^2=-y_1\biggl[x_1^2\sum_{w=0}^{x_1}w^2-2x_1\sum_{w=0}^{x_1}w^3+\sum_{w=0}^{x_1}w^4\biggr]\\
    &=-y_1\biggl[\frac{1}{6}x_1^3(x_1+1)(2x_1+1)-\frac{1}{2}x_1^3(x_1+1)^2+\frac{1}{30}(6x_1^5+15x_1^4+10x_1^3-x_1)\biggr]\\
    &=\frac{1}{30}y_1x_1(1-x_1^4).
\end{align}
{Analogously}

\begin{align}
    &F_{G_1,\textbf{c},r_2-l}(x_2,x_1,y_1)=\frac{1}{30}y_1x_2(1-x_2^4)\\
    &F_{G_1,\textbf{c},r_1-l}(x_1,x_2,y_1)=\frac{1}{30}y_1(x_1+k)(1-(x_1+k)^4)\\
    &F_{G_1,\textbf{c},r_1-l}(x_2,x_1,y_1)=\frac{1}{30}y_1(x_2+k)(1-(x_2+k)^4)\\
    &F_{G_2,\textbf{c},r_1-l}(x_1,x_2,y_1)=\frac{1}{30}y_1^2(y_1^4-1)
\end{align}
{Therefore}

\begin{equation}
    {F^{n_1,n_2}_{(\textbf{d}^r,\textbf{d}^l),\textbf{c},1}(x_1,x_2,y_1)=\sum_{i=1}^2(F_{G_1,\textbf{c},r_i-l}(x_1,x_2,y_1)+F_{G_1,\textbf{c},r_i-l}(x_2,x_1,y_1))+F_{G_2,\textbf{c},r_1-l}(x_1,x_2,y_1)}
\end{equation}
{We see that the degree of $F^{n_1,n_2}_{(\textbf{d}^r,\textbf{d}^l),\textbf{c},1}(x_1,x_2,y_1)$ is $n_2+3g+2a-2=6$, which is the claim of \cref{thm2}.}

{Finally, we prove that $F^{n_1,n_2}_{(\textbf{d}^r,\textbf{d}^l),\textbf{c},1}(x_1,x_2,y_1)$ is even. For the purpose of illustrating the process, we will establish that $F_{G_1,\textbf{c},r_2-l}(x_1,x_2,y_1)$ is even. The same reasoning can be applied to the remaining functions, thus proving that $F^{n_1,n_2}_{(\textbf{d}^r,\textbf{d}^l),\textbf{c},1}(x_1,x_2,y_1)$ is even. The lattice points of the flow polytope will be}

\begin{equation}
    {W_{G_1,\textbf{c},r_2-l}(x_1,x_2,y_1)=\{(x_1,x_2;x_1-w,x_1-w,w,w;-y_1)|x_1,x_2>0,y_1<0,0<w<x_1\}\cap\mathbb Z^7}
\end{equation}
{Then}

\begin{equation}
    {W_{G_1,\textbf{c},r_2-l}(-x_1,-x_2,-y_1)=\{(-x_1,-x_2;w-x_1,w-x_1,-w,-w;y_1)|x_1,x_2<0,y_1>0,x_1<w<0\}\cap\mathbb Z^7}
\end{equation}
{Hence}

\begin{align}
    F_{G_1,\textbf{c},r_2-l}(-x_1,-x_2,-y_1)&=\sum_{w=-x_1}^{0}y_1(w-x_1)^2(-w)^2=\sum_{w=-x_1}^{0}y_1(x_1-w)^2w^2\\
    &=\sum_{w=0}^{x_1}(-y_1)(x_1-w)^2w^2=F_{G_1,\textbf{c},r_2-l}(x_1,x_2,y_1)
\end{align}

\section{Double Gromov--Witten invariants via the bosonic Fock space}
\label{sec-fockspace}
In this section, we establish a connection between our invariants $N_{\textbf c,g}^{\alpha,\beta,\tilde\alpha,\tilde\beta,\bullet}(\textbf d)$ and the bosonic Fock space. More precisely, we express them as so-called \textit{matrix elements} of certain operators. For this, we first need to re-organise floor diagrams. In this, we follow the notion of floor diagram introduced in \cite{cavalieri2021counting} for the study of curves in Hirzebruch surfaces and fix the same setting as for \cref{def-floordiag}.

\begin{definition}
    Let $\mathcal{F}$ a loop-free graph with vertex set $V$, edge set $E$. There are two types of edges: bounded edges which are composed of two half-edges adjacent to different vertices and unbound edges called \textit{ends} with one flag. We call $\mathcal{F}$ a \textbf{thickened floor diagram} of multidegree $\textbf{d}$ and relative to $(\textbf{x},\textbf{y})$ if:
    \begin{itemize}
        \item Each vertex $v\in V$ carries a size $s_v$ that is either $0$ or $1$.
        \item Each half-edge may be decorated with a thickening and for each bounded edge exactly one of its half-edges is thickened.
        \item At each vertex $v$, exactly $2-2s_v$ half-edges are thickened.
        \item We have a map
        \begin{equation}
            w\colon E\to\mathbb{Z}_{>0}
        \end{equation}
        associating to each edge $e\in E$ an expansion factor $w(e)$.\\
        \item For each vertex $v$ denote by $E_v^-$ the set of incoming edges at $v$ and by $E_v^+$ the set of outgoing edges at $v$. We define the divergence of $v$ as
        \begin{equation}
            \mathrm{div}(v)=\sum_{e\in E_v^+}w(e)-\sum_{e\in E_v^-}w(e).
        \end{equation}
        Then, we require that 
        \begin{itemize}
            \item if $s_v=0$, then $\mathrm{div}(v)=0$
            \item if $s_v=1$, then $\mathrm{div}(v)=r_i-l_i$, where $v$ is the $i$-th vertex with $s_v=1$.
        \end{itemize}
        \item The sequence of expansion factors of non-thick ends is given by $\textbf{x}$, while the sequence of expansion factors of thick ends is given by $\textbf{y}$, where negative entries of $\textbf{x}$ and $\textbf{y}$ correspond to ends pointing to the left and positive entries to ends pointing to the right.
        \item The ends of $\mathcal{F}$  are labelled by the parts of $\textbf{x}$ and $\textbf{y}$.
    \end{itemize}

    The genus of $\mathcal{F}$ is the first Betti number of the underlying graph.\\
    Finally, we denote by
    \begin{equation}
        \mu(\mathcal{F})=\prod w(e),
    \end{equation}
    where the product runs over all bounded edges of $\mathcal{F}$, the \textbf{weight} of $\mathcal{F}.$
\end{definition}

\begin{example}
    Let us consider the data from \cref{example-compute_invariant}. In \cref{fig:thickened_floor_diagrams}, we provide all thickened floor diagrams satisfying these data.
    \begin{figure}
        \centering
        \tikzset{every picture/.style={line width=0.75pt}} 



        \caption{}
        \label{fig:thickened_floor_diagrams}
    \end{figure}
\end{example}

Similarly to \cref{thm4}, thickened floor diagrams compute our invariants $N_{\textbf c,g}^{\alpha,\beta,\tilde\alpha,\tilde\beta,\bullet}(\textbf d)$. More precisely, we have the following theorem which follows from the same arguments as in \cite[section 5]{cavalieri2021counting}.

\begin{theorem}\label{thm-thickfloor}
    Let $\textbf d=(d^t;\textbf d^r;\textbf d^l)$ be a vector of positive integer numbers, $g\geq0$ an integer and \textbf{x} a vector with coordinates in $\mathbb{Z}\setminus\{0\}$. We write $\alpha(\textbf x)=\alpha$ and $\tilde\alpha(\textbf x)=\tilde\alpha$. Then, for any two sequences of non-negative integer numbers $\beta=(\beta_i)_{i\geq1}$ and $\tilde\beta=(\tilde\beta_i)_{i\geq1}$ such that
    \begin{equation}
        \sum_ii(\alpha_i+\beta_i)=d^t+\sum_{i=1}^nc_i^rd_i^r-\sum_{j=1}^mc_j^ld_j^l\quad\text{and}\quad\sum_ii(\tilde\alpha_i+\tilde\beta_i)=d^t,
    \end{equation}
    one has
    \begin{equation}
        N_{\textbf c,g}^{\alpha,\beta,\tilde\alpha,\tilde\beta,\bullet}(\textbf d)=\sum_{\mathcal F}\mu(\mathcal F)
    \end{equation}
    where the sum runs over all thickened floor diagrams $\mathcal F$ of multidegree $\textbf d$, genus $g$ and relative to $(\textbf{x},\textbf{y})$ for $S(\textbf{c})$.
\end{theorem}

This interpretation of $N_{\textbf c,g}^{\alpha,\beta,\tilde\alpha,\tilde\beta,\bullet}(\textbf d)$ is the basis for the remainder of the section. Next, we introduce the bosonic Fock space and related notions.

\begin{definition}
    We define the two-dimensional Heisenberg algebra $\mathcal{H}$ as generated by $(a_n)_{n\in\mathbb{Z}}$ and $(b_n)_{n\in\mathbb{Z}}$ satisfying the commutator relations
    \begin{equation}
        [a_n,a_m]=0,\quad [b_n,b_m]=0\quad\textrm{and}\quad [a_n,b_m]=n\delta_{n,-m}
    \end{equation}
    where $\delta_{n,-m}$ is the Kronecker symbol. Moreover, we set $a_0=b_0=0$.
\end{definition}

We now consider the free action of $\mathcal{H}$ on the so-called vaccuum vector $v_{\emptyset}$ where we set $a_n\cdot v_{\emptyset}=b_n\cdot v_{\emptyset}=0$ for $n>0$. For two partitions $\mu,\nu$, we define
\begin{equation}
    v_{\mu,\nu}=\frac{1}{|\mathrm{Aut}(\mu)||\mathrm{Aut}(\nu)|}\prod_{i=1}^{\ell(\mu)}a_{-\mu_i}\prod_{j=1}^{\ell(\nu)}b_{-\nu_j}\cdot v_{\emptyset}.
\end{equation}

The vectors $v_{\mu,\nu}$ form a basis of a vector space which we call the bosonic Fock space. We also define an inner product by declaring $\langle v_{\emptyset}\mid v_{\emptyset}\rangle=1$, $a_{n}$ as the adjoint of $a_{-n}$ and $b_{n}$ as the adjoint of $b_{-n}$. Thus, we obtain
\begin{equation}
    \langle v_{\mu,\nu}\mid v_{\mu',\nu'}\rangle=\frac{1}{|\mathrm{Aut}(\mu)|}\frac{1}{|\mathrm{Aut}(\nu)|}\prod\mu_i\prod\nu_i\delta_{\mu,\nu'}\delta_{\mu',\nu}.
\end{equation}
Moreover, for an operator $M\in\mathcal{H}$, we denote $\langle v\mid M\mid w\rangle=\langle v\mid Mw\rangle$. For partitions $\mu,\nu,\mu',\nu'$, we call $\langle v_{\mu,\nu}\mid M\mid v_{\mu',\nu'}\rangle$ a \textbf{matrix element} of $M$. Moreover, in the special case $\langle v_\emptyset\mid M\mid v_\emptyset\rangle$, we call this expression a \textbf{vaccuum expectation}.

Finally, we define the \textbf{normal ordering} of a monomial of operators $\colon\prod\alpha_i\prod\beta_j\colon$ as the product, such that all $\alpha_i,\beta_j$ with $i,j<0$ appear to left of the operators with positive index.\\

We now define the operator whose matrix elements will give our invariants.

\begin{definition}
We fix two formal variables $u$ and $t$. We first define
\begin{equation}
    \textbf{a}_n\coloneqq\begin{cases}
        ua_n,\,\textrm{if}\,n<0\\
        a_n\,\textrm{if}\,n>0
    \end{cases}\quad\textrm{and}\quad 
    \textbf{b}_n\coloneqq\begin{cases}
        ub_n,\,\textrm{if}\,n<0\\
        b_n\,\textrm{if}\,n>0
    \end{cases}
\end{equation}

Moreover, we define

\begin{equation}
    M_c=u^{-1}t\sum_{m\in\mathbb{Z}_{>0}}\sum_{\substack{\textbf{z}\in(\mathbb{Z}\backslash\{\mathbf{0}\})^m\colon\\\sum z_i=-c}}\colon \textbf{a}_{z_1}\cdots\textbf{a}_{z_m}\colon
\end{equation}
and
\begin{equation}
    M=u^{-1}\sum_{m\in\mathbb{Z}_{>0}}\colon \textbf{b}_{-m}\textbf{b}_m\colon.
\end{equation}
\end{definition}

We are now ready to state our main theorem of this section. For this, we denote by $\textbf{x}^{\pm},\textbf{y}^{\pm}$ the vectors of respectively positive or negative entries of $\textbf{x},\textbf{y}$.

\begin{theorem}
\label{thm:fockspace}
    Consider the same data as in \cref{thm1}. Then, we have
    \begin{align}     N_{\textbf{c},g}^{\alpha,\beta,\tilde{\alpha},\tilde{\beta},\bullet}(\textbf{d})=&\frac{|\mathrm{Aut}(\textbf{x})||\mathrm{Aut}(\textbf{y})|}{\prod|x_i|\prod|y_j|}\\
    &\left\langle v_{\textbf{y}^-,\textbf{x}^-}\,\bigg\vert \left[t^au^{g-1+\ell(\textbf{x}^-)+\ell(\textbf{y}^-)}\right]\sum_{\sigma}\sum_{\substack{(r_1,\dots,r_a)\\(l_1,\dots,l_a)}}\sigma\left(\prod_{i=1}^a(M_{r_i-l_i}M^{a+g+\ell(\textbf{y})-1})\right)\bigg\vert \,v_{\textbf{y}^+,\textbf{x}^+}\right\rangle,
    \end{align}
    where $[t^au^{g-1+\ell(\textbf{x}^-)+\ell(\textbf{y}^-)}]$ denotes the coefficient of the monomial $t^{a}u^{g-1+\ell(\textbf{x}^-)+\ell(\textbf{y}^-)}$ and the sum runs over all permutations of $D_r$ and $D_l$. Moreover the first sum runs over all permutations $\sigma\in S_{3a+g+\ell(\textbf{y})-1}$ that respect the ordering of $1,\dots,a$ and $\sigma\left(\prod_{i=1}^aM_{r_i-l_i}M^{a+g+\ell(\textbf{y})-1}\right)$ denotes the permutations of the $2a+g+\ell(\textbf{y})-1$ factors induced by $\sigma$.
\end{theorem}

The proof follows from a tropical formulation of Wick's theorem \cite{wick1950evaluation} that may be found in \cite[Proposition 5.2]{block2016fock} and in \cite[Proposition 6.6]{cavalieri2021counting} for the version we employ. The key idea is to construct Feynman diagrams corresponding to matrix elements of monomials in the bosonic Fock space. 
The idea, also outlined after \cite[Definition 6.4]{cavalieri2021counting} is as follows: Let $m_-,m_1,\dots,m_s,m_+$ be normally ordered monomials in the operators $a_i,b_j$, such that $m_-$ contains only operators with negative index and $m_+$ only operators with positive index. Then, we associate to the monomial $P=m_+m_1\cdots m_sm_-$ a family of graphs, we call \textbf{Feynman diagrams} associated to $P$ via the following procedure:
\begin{enumerate}
    \item For each monomial $m_l$, we create a vertex $v_l$. For each $a_i$ with $i<0$ appearing in $m_l$, we create a half-edge $e$ adjacent to $v_l$ pointing to the left with expansion factor $w(e)=|i|$. For each $a_i$ with $i>0$ we do create a half-edge pointing to the right with expansion factor $w(e)=i$. Moreover, for each $b_j$ appearing in $m_l$, we create a thickened half-edge adjacent to $v_l$ again pointing to the left for $j<0$ with expansion factor $w(e)=|j|$ and to the right for $j>0$ with expansion factor $w(e)=j$.\\
    For the monomial $m_-$ and $m_+$ we create a set of unbounded half-edges of corresponding weight for each operator appearing in them. We thicken the half-edges coming from operators $b_i$.
    \item We order all pieces by moving the half-edges corresponding to $m_+$ to the left and the half-edges corresponding to $m_-$ to the right. All vertices $m_l$ are ordered linearly respecting the ordering of the indices $l$.
    \item We connect half-edges to each other according to the following rules:
    \begin{itemize}
        \item A half-edge pointing to the left is connected to a half-edge pointing to the right and the connection respects the ordering of the vertices.
        \item Connected half-edges must have the same expansion factor.
        \item All resulting edges that adjacent to two vertices have exactly one thickened half-edge. 
    \end{itemize}
\end{enumerate}

    \begin{figure}
        \centering
        \tikzset{every picture/.style={line width=0.75pt}} 

\tikzset{every picture/.style={line width=0.75pt}} 

\begin{tikzpicture}[x=0.75pt,y=0.75pt,yscale=-0.75,xscale=0.75]

\draw [line width=4.5]    (803.01,661.87) -- (724.01,662.32) ;
\draw  [fill={rgb, 255:red, 0; green, 0; blue, 0 }  ,fill opacity=1 ] (592,661.55) .. controls (592.03,665.7) and (588.69,669.07) .. (584.54,669.1) .. controls (580.4,669.12) and (577.03,665.78) .. (577,661.64) .. controls (576.98,657.5) and (580.32,654.12) .. (584.46,654.1) .. controls (588.6,654.07) and (591.98,657.41) .. (592,661.55) -- cycle ;
\draw  [fill={rgb, 255:red, 0; green, 0; blue, 0 }  ,fill opacity=1 ] (452.01,662.34) .. controls (452.03,666.48) and (448.69,669.85) .. (444.55,669.88) .. controls (440.41,669.9) and (437.03,666.56) .. (437.01,662.42) .. controls (436.98,658.28) and (440.32,654.9) .. (444.46,654.88) .. controls (448.61,654.85) and (451.98,658.19) .. (452.01,662.34) -- cycle ;
\draw  [fill={rgb, 255:red, 0; green, 0; blue, 0 }  ,fill opacity=1 ] (332,662.01) .. controls (332.02,666.15) and (328.69,669.52) .. (324.54,669.55) .. controls (320.4,669.57) and (317.02,666.23) .. (317,662.09) .. controls (316.98,657.95) and (320.32,654.57) .. (324.46,654.55) .. controls (328.6,654.52) and (331.98,657.86) .. (332,662.01) -- cycle ;
\draw  [fill={rgb, 255:red, 0; green, 0; blue, 0 }  ,fill opacity=1 ] (212,662.68) .. controls (212.03,666.82) and (208.69,670.19) .. (204.55,670.22) .. controls (200.4,670.24) and (197.03,666.9) .. (197,662.76) .. controls (196.98,658.62) and (200.32,655.24) .. (204.46,655.22) .. controls (208.6,655.19) and (211.98,658.53) .. (212,662.68) -- cycle ;
\draw    (661,660.67) -- (592,661.55) ;
\draw    (577,661.64) -- (535.15,687.37) ;
\draw    (577,661.64) -- (535.88,639.37) ;
\draw [line width=4.5]    (487,661.64) -- (452.01,662.34) ;
\draw [line width=4.5]    (437.01,662.42) -- (402.01,663.11) ;
\draw [line width=4.5]    (367,661.31) -- (332,662.01) ;
\draw [line width=4.5]    (317,662.09) -- (282,662.78) ;
\draw    (244.01,663) -- (212,662.68) ;
\draw    (197,662.76) -- (165,662.44) ;
\draw    (116.79,624.71) -- (47.79,625.59) ;
\draw [line width=4.5]    (121.23,703.68) -- (42.23,704.12) ;

\draw (753.11,669.11) node [anchor=north west][inner sep=0.75pt]    {$1$};
\draw (619.08,663.86) node [anchor=north west][inner sep=0.75pt]    {$1$};
\draw (63.68,594.96) node [anchor=north west][inner sep=0.75pt]    {$1$};
\draw (459.09,666.75) node [anchor=north west][inner sep=0.75pt]    {$1$};
\draw (415.29,666.86) node [anchor=north west][inner sep=0.75pt]  [rotate=-1]  {$1$};
\draw (72.33,709.91) node [anchor=north west][inner sep=0.75pt]    {$5$};
\draw (173.08,665.35) node [anchor=north west][inner sep=0.75pt]    {$5$};
\draw (221.09,667.08) node [anchor=north west][inner sep=0.75pt]    {$3$};
\draw (293.11,669.68) node [anchor=north west][inner sep=0.75pt]    {$3$};
\draw (341.29,667.27) node [anchor=north west][inner sep=0.75pt]  [rotate=-1]  {$3$};
\draw (546.16,678.27) node [anchor=north west][inner sep=0.75pt]    {$3$};
\draw (549.86,624.24) node [anchor=north west][inner sep=0.75pt]    {$1$};

\end{tikzpicture}

        \caption{}
        \label{fig:fragments}
    \end{figure}

    \begin{figure}
        \centering
        \tikzset{every picture/.style={line width=0.75pt}} 

\tikzset{every picture/.style={line width=0.75pt}} 

\begin{tikzpicture}[x=0.75pt,y=0.75pt,yscale=-0.75,xscale=0.75]

\draw [line width=4.5]    (803.99,883.06) -- (724.99,882.84) ;
\draw  [fill={rgb, 255:red, 0; green, 0; blue, 0 }  ,fill opacity=1 ] (593,880.97) .. controls (592.98,885.12) and (589.62,888.46) .. (585.48,888.45) .. controls (581.33,888.44) and (577.99,885.07) .. (578,880.93) .. controls (578.01,876.79) and (581.38,873.44) .. (585.52,873.45) .. controls (589.66,873.46) and (593.01,876.83) .. (593,880.97) -- cycle ;
\draw  [fill={rgb, 255:red, 0; green, 0; blue, 0 }  ,fill opacity=1 ] (453,880.58) .. controls (452.99,884.72) and (449.62,888.07) .. (445.48,888.06) .. controls (441.33,888.05) and (437.99,884.68) .. (438,880.54) .. controls (438.01,876.4) and (441.38,873.05) .. (445.52,873.06) .. controls (449.66,873.07) and (453.01,876.44) .. (453,880.58) -- cycle ;
\draw  [fill={rgb, 255:red, 0; green, 0; blue, 0 }  ,fill opacity=1 ] (333,879.25) .. controls (332.99,883.39) and (329.62,886.74) .. (325.48,886.73) .. controls (321.34,886.71) and (317.99,883.35) .. (318,879.21) .. controls (318.01,875.06) and (321.38,871.71) .. (325.52,871.73) .. controls (329.66,871.74) and (333.01,875.11) .. (333,879.25) -- cycle ;
\draw  [fill={rgb, 255:red, 0; green, 0; blue, 0 }  ,fill opacity=1 ] (213,878.91) .. controls (212.99,883.05) and (209.62,886.4) .. (205.48,886.39) .. controls (201.34,886.38) and (197.99,883.01) .. (198,878.87) .. controls (198.01,874.73) and (201.38,871.38) .. (205.52,871.39) .. controls (209.66,871.4) and (213.01,874.77) .. (213,878.91) -- cycle ;
\draw    (662,880.67) -- (593,880.97) ;
\draw    (578,880.93) -- (535.93,906.31) ;
\draw    (578,880.93) -- (537.06,858.32) ;
\draw [line width=4.5]    (488,880.18) -- (453,880.58) ;
\draw [line width=4.5]    (438,880.54) -- (403,880.94) ;
\draw [line width=4.5]    (368,878.85) -- (333,879.25) ;
\draw [line width=4.5]    (318,879.21) -- (283,879.61) ;
\draw    (245,879.5) -- (213,878.91) ;
\draw    (198,878.87) -- (166,878.28) ;
\draw    (118.11,840.15) -- (49.11,840.45) ;
\draw [line width=4.5]    (121.89,919.16) -- (42.89,918.94) ;
\draw  [dash pattern={on 4.5pt off 4.5pt}]  (724.99,882.84) -- (662,880.67) ;
\draw  [dash pattern={on 4.5pt off 4.5pt}]  (537.06,858.32) -- (488,880.18) ;
\draw  [dash pattern={on 4.5pt off 4.5pt}]  (403,880.94) .. controls (407.12,837.95) and (178.11,839.31) .. (118.11,840.15) ;
\draw  [dash pattern={on 4.5pt off 4.5pt}]  (166,878.28) -- (121.89,919.16) ;
\draw  [dash pattern={on 4.5pt off 4.5pt}]  (283,879.61) -- (245,879.5) ;
\draw  [dash pattern={on 4.5pt off 4.5pt}]  (535.93,906.31) .. controls (485.86,930.17) and (429.81,946.02) .. (368,878.85) ;

\draw (753.94,889.95) node [anchor=north west][inner sep=0.75pt]    {$1$};
\draw (619.96,883.57) node [anchor=north west][inner sep=0.75pt]    {$1$};
\draw (65.16,810.02) node [anchor=north west][inner sep=0.75pt]    {$1$};
\draw (459.95,885.12) node [anchor=north west][inner sep=0.75pt]    {$1$};
\draw (415.95,885) node [anchor=north west][inner sep=0.75pt]    {$1$};
\draw (72.84,925.04) node [anchor=north west][inner sep=0.75pt]    {$5$};
\draw (173.96,881.33) node [anchor=north west][inner sep=0.75pt]    {$5$};
\draw (221.96,883.46) node [anchor=north west][inner sep=0.75pt]    {$3$};
\draw (293.95,886.66) node [anchor=north west][inner sep=0.75pt]    {$3$};
\draw (341.95,884.79) node [anchor=north west][inner sep=0.75pt]    {$3$};
\draw (546.92,897.37) node [anchor=north west][inner sep=0.75pt]    {$3$};
\draw (551.07,843.38) node [anchor=north west][inner sep=0.75pt]    {$1$};

\end{tikzpicture}

        \caption{}
        \label{fig:Feynman}
    \end{figure}

    \begin{example}\label{example-construction}
    Let $P$ be the following monomial
    \begin{equation}
        P=(b_5a_{1})\cdot(a_{-5}a_3)\cdot(b_{-3}b_3)\cdot(b_{-1}b_1)\cdot(a_{-3}a_{-1}a_1)\cdot b_{-1}.
    \end{equation}
    The factors $m_i$ are separated by parentheses. Moreover, $m_+=b_5a_1$ and $m_-=b_{-1}$. \cref{fig:fragments} depicts the first point of the procedure outlined above, while \cref{fig:Feynman} provides the unique Feynman graph that respects the procedure. Note that after dropping all external half--edges in \cref{fig:Feynman} and flipping the graph left--to--right, we obtain the thickened floor diagram on the top-left in \cref{fig:thickened_floor_diagrams}. This illustrates the general correspondence.
\end{example}

As such the monomials in the operators $M_c$ corresponds to vertices of size $1$ and non-trivial divergence, whereas the monomials in the operator $M$ yield vertices of size $0$ and trivial divergence. Note that by the construction of thickened Floor diagrams in \cite{cavalieri2021counting}, all vertices of size $0$ are $2$-valent.

\begin{proposition}[{\cite[Proposition 5.2]{block2016fock},\cite[Proposition 6.6]{cavalieri2021counting}}]\label{prop-Wick}
    The vaccum expecation of an operator $P$ as above is equal to the weighted sum of all Feynman diagrams associated to $P$, where each diagram is weighted the product of all edges (bounded and unbounded).
\end{proposition}

\begin{example}
    We compute $N_{(2;-1,0),0}^{00001,1,1,0}(1;2;1,1)$ as in \cref{example-compute_invariant}. We have that
    \begin{align}
        N_{(2;-1,0),0}^{00001,1,1,0}(1;2;1,1)&=\frac{1}{5}\langle a_1b_5v_\emptyset|[t^2u^2](2M_2MMM_3+M_2MM_3M+2M_3MMM_2+M_3MM_2M)|b_1v_\emptyset\rangle=\\
    &=\frac{1}{5}\biggl(\langle a_1b_5a_{-5}a_3b_{-3}b_3b_{-1}b_1a_{-3}a_{-1}a_1b_{-1}\rangle+\langle a_1b_5a_{-5}a_3b_{-1}b_1b_{-3}b_3a_{-3}a_{-1}a_1b_{-1}\rangle+\\
    &+\langle a_1b_5b_{-1}b_1a_{-5}a_3b_{-3}b_3a_{-3}a_{-1}a_1b_{-1}\rangle+\langle a_1b_5b_{-1}b_1a_{-5}a_{-1}a_4b_{-4}b_4a_{-4}a_1b_{-1}\rangle+\\
    &+\langle a_1b_5a_{-5}a_2b_{-2}b_2b_{-1}b_1a_{-2}a_{-1}a_1b_{-1}\rangle+\langle a_1b_5a_{-5}a_2b_{-1}b_1b_{-2}b_2a_{-2}a_{-1}a_1b_{-1}\rangle+\\
    &+\langle a_1b_5b_{-1}b_1a_{-5}a_2b_{-2}b_2a_{-2}a_{-1}a_1b_{-1}\rangle+\langle a_1b_5b_{-1}b_1a_{-1}a_{-5}a_3b_{-3}b_3a_{-3}a_1b_{-1}\rangle\biggr).
    \end{align}
    After a careful calculation, we obtain the following:
    \begin{equation}
        N_{(2;-1,0),0}^{00001,1,1,0}(1;2;1,1)=\frac{1}{5}(3\cdot45+80+3\cdot20+45)=\frac{320}{5}=64.
    \end{equation}
\end{example}

We end this section with the proof of \cref{thm:fockspace}.

\begin{proof}[Proof of \cref{thm:fockspace}]
    Clearly all Feynman diagrams are thickend Floor diagrams and the operators $M_c$ and $M$ are built, such that all thickend floor diagrams contributing to $N_{\textbf{c},g}^{\alpha,\beta,\tilde{\alpha},\tilde{\beta},\bullet}(\textbf{d})$ appear. The weights only differ by the product of the expansion factors of the ends, which is why we divide by them. A simple Euler characteristic calculation shows that all Feynman diagrams are of the right genus.
\end{proof}

\printbibliography
\end{document}